\documentclass[11pt,a4paper]{amsart}[2000]
\usepackage{latexsym,amsfonts,amssymb,epsfig,verbatim}
\usepackage{amsmath, latexsym, graphics, textcomp}
\usepackage{amsthm}
\usepackage{pinlabel}
\usepackage{mathtools}
\usepackage{mathrsfs}
\usepackage{xcolor}
\usepackage{enumitem}
\usepackage{tikz}
\usepackage{hyperref}
\usepackage{graphicx}
\usepackage{csquotes}
\usetikzlibrary{decorations.markings}

\usepackage[left=3cm,right=3cm, bottom=3cm]{geometry}

\usepackage[all,2cell,dvips]{xy}

\newcommand{\nc}{\newcommand}
\nc{\dmo}{\DeclareMathOperator}
\nc{\nt}{\newtheorem}
\theoremstyle{definition}

\nt{theorem}{Theorem}[section]
\nt{lemma}[theorem]{Lemma}
\nt{prop}[theorem]{Proposition}
\nt{definition}[theorem]{Definition}
\nt{cor}[theorem]{Corollary}
\nt{example}[theorem]{Example}
\nt{remark}[theorem]{Remark}
\nt{notation}[theorem]{Notation}
\nt{fact}[theorem]{Fact}
\nt{question}[theorem]{Question}
\nt{conjecture}[theorem]{Conjecture}

\nc{\Z}{\mathbb{Z}}
\nc{\R}{\mathbb{R}}
\nc{\Q}{\mathbb{Q}}

\nc{\margin}[1]{\marginpar{\tiny #1}}

\nc{\p}[1]{\smallskip\noindent{{\bf #1}}}

\begin{document}
\include{diagram}

\title{The isomorphism problem for small-rose generalized Baumslag-Solitar groups}

\author{Daxun Wang}
\address{Department of Mathematics,
		University at Buffalo - SUNY,
		Buffalo, NY, 14260}
\email{daxunwan@buffalo.edu}

\keywords{Generalized Baumslag-Solitar groups, the isomorphism problem, deformation space of trees.}

\subjclass[2000]{20E08, 20F10, 20F65.}

\begin{abstract}
A generalized Baumslag-Solitar group is a finitely generated group that acts on a tree with infinite-cyclic vertex and edge stabilizers. In this paper, we show that the isomorphism problem is solvable for small-rose non-ascending generalized Baumslag-Solitar groups.
\end{abstract}

\maketitle

\section{Introduction}
In 1912, Dehn formulated three fundamental decision problems of groups: the word problem, the conjugacy problem and the isomorphism problem. The word problem and the conjugacy problem for a group address how to decide whether two words in some finite generating set represent the same or conjugate elements. On the other hand, the isomorphism problem consists of determining whether two given group presentations define isomorphic groups. The answers to all of these problems, either positive or negative, have a major impact on studying structure as well as topological properties of groups.

In this paper, we are particularly interested in the isomorphism problem of  \textit{generalized Baumslag-Solitar groups} (GBS) groups. These groups are finitely generated groups acting on a tree with infinite-cyclic vertex and edge stabilizers. Basic examples of GBS groups include the Baumslag-Solitar groups BS($m,n$)$=\langle x,t\ |\ tx^mt^{-1}=x^n\rangle$, as well as the knot group $\langle x,y\ |\ x^p=y^q\rangle$ of a $(p,q)$-torus knot. GBS groups have been studied in many aspects such as JSJ decomposition of groups \cite{M03}, splitting of groups \cite{M06}, cohomological dimension \cite{K90'}, the automorphism group \cite{G07}, \cite{C09} and one-relator groups \cite{Mc91}. Some subclasses of GBS groups were classified up to quasi-isometry in \cite{FM98} and \cite{W01}.

The isomorphism problem has only been shown to be solvable in special cases. Forester solved the isomorphism problem for GBS groups with no non-trivial integral moduli \cite{M06}. Levitt showed that the isomorphism problem is solvable for GBS groups $G$ such that Out($G$) does not contain a non-abelian free group \cite{G07}. Clay and Forester solved the case for GBS groups whose labeled graphs have first Betti number at most one \cite{MM08}. Dudkin solved the case for GBS groups where one of the labeled graphs has a sole mobile edge \cite{FA17}.

Recently, a particular subclass of GBS groups called the \textit{$n$-rose GBS groups}, has been studied in relation with boundary actions and $C^{\ast}$-simplicity \cite{MW22}. The definition of these groups is as follows: 
\begin{definition}[$n$-rose GBS groups]
$$G=\langle x, t_1,..., t_n\ |\ t_ix^{p_i}t_i^{-1}=x^{q_i}, \mbox{for some}\ p_i,q_i\in \mathbb{Z}\setminus\{0\}, i=1,...,n \rangle$$
\end{definition}

By Bass-Serre theory \cite{BK90}, these groups can be realized as the fundamental group of a labeled graph $\Gamma$ such that $\Gamma$ is a wedge sum of $n$ loops and each loop carries two labels $p_i$ and $q_i$ at its endpoints. $\Gamma$ is then called a \textit{$n$-rose graph}. For the isomorphism problem of $n$-rose GBS groups, Casals-Ruiz, Kazachkov and Zakharov solved the case where $p_i=1$ and $q_i=q_j$ for all $i,j$ \cite{CRKZ21}. However, these groups are a family of ascending GBS groups. Motivated by their results, our goal is to study the isomorphism problem for non-ascending $n$-rose GBS groups.

Our first main result is Theorem 1.2. We show that:
\begin{theorem}
Let $G$ be a $n$-rose GBS group satisfying either one of the following conditions:
\begin{itemize}[topsep=0pt, itemsep=1ex]
\item[(a)] the underlying $n$-rose graph $\Gamma$ has only one mobile edge; or
\item[(b)] $n\leq 3$.
\end{itemize} 
Then there is an algorithm to determine whether $G$ is ascending.
\end{theorem}

Forester proved that any two reduced labeled graphs representing the same GBS group $G$ are contained in the canonical deformation space of $G$ \cite{M02}. In addition, if $G$ is non-ascending, there is a sequence of slide moves taking one labeled graph to the other. Our main technical result, Proposition \ref{prop: e commute}, shows that if $G$ is a non-ascending $n$-rose GBS group with $n\leq 3$, then we can rearrange this slide sequence by switching slide moves of \textit{mobile edges}, which are introduced by Clay and Forester in \cite{MM08}. From these we deduce our second main result Theorem 1.3. We show that:

\begin{theorem} The isomorphism problem is solvable for non-ascending $n$-rose GBS groups with $n\leq 3$.

Example \ref{example: 3-rose non-ascending} exhibits an infinite family of non-ascending 3-rose GBS groups to which this Theorem applies.
\end{theorem}

\textbf{Acknowledgements.} This research was supported by travel funding from National Science Foundation Grant No. DMS–1812021. I would like to thank my advisor Johanna Mangahas for her exceptional guidance and encouragement. I would also like to thank Matt Clay, F.A. Dudkin and Max Forester for helpful discussions.

\section{Preliminaries}

\subsection{Generalized Baumslag-Solitar group}
A \textit{graph} $\Gamma=(V(\Gamma),E(\Gamma))$ is a pair of sets where $V(\Gamma)$ are the vertices and $E(\Gamma)$ are the edges. An edge comes with two maps $o,t$ which associate its initial and terminal vertex, respectively. In addition, there is a fixed-point-free involution $e\mapsto \overline{e}$ of $E(\Gamma)$ which reverses the orientation of edges.

A \textit{graph of groups} is a graph $\Gamma$ such that we associate to every vertex $v\in V(\Gamma)$ a group $G_v$ and, for every edge $e\in E(\Gamma)$, a group $G_e$ together with an injective homomorphism $\phi_e:G_e\hookrightarrow G_{t(e)}$. Furthermore, we require $G_e=G_{\overline{e}}$ and thus have $\phi_{\overline{e}}:G_{e}\hookrightarrow G_{o(e)}$. A \textit{G-tree} is a simplicial tree $T$ together with an action of $G$ by automorphism without inversion of edges. By Bass-Serre theory \cite{BK90}, the quotient graph $T/G$ has a graph of groups structure with fundamental group $G$.

A \textit{generalized Baumslag-Solitar group} (GBS group) is a finitely generated group $G$ which acts on a tree with all vertex and edge stabilizers infinite-cyclic. The tree is called a \textit{GBS tree}. In the quotient graph of groups, every vertex and edge group is isomorphic to $\mathbb{Z}$. If we choose generators for vertex and edge groups, each inclusion map $G_e \hookrightarrow G_{o(e)}$ is given by multiplication by a non-zero integer, which we denote by $\lambda(e)$. Thus any such quotient graph of groups can be realized by a labeled graph. Specifically, a \textit{labeled graph} is a finite graph $\Gamma$ where each oriented edge $e$ has a label $\lambda(e)\in \mathbb{Z}\setminus \{0\}$ at its origin. Notice that if we replace a generator of an edge group $G_e$ by its inverse, it changes the signs of $\lambda(e)$ and $\lambda(\overline{e})$. If we replace a generator of a vertex group $G_v$ by its inverse, it changes the signs $\lambda(e)$ for all edges $e$ emanating from $v$. These are called \textit{admissible sign changes}. By Remark 2.3 in \cite{M06}, given a GBS tree $T$, the quotient labeled graph $\Gamma=T/G$ is well-defined up to admissible sign changes.

A GBS group is \textit{non-elementary} if it is not isomorphic to one of the following groups: $\mathbb{Z}$, $\mathbb{Z}^2$ or the Klein bottle group.

\subsection{Deformation space and deformation moves}
Given a GBS tree $T$, an edge $e$ is \textit{collapsible} if $G_e=G_{o(e)}$ and its endpoints are not in the same orbit. If we collapse every edge in the orbit of $e$ to a vertex, we obtain a new tree $T'$. We say that $T'$ is obtained from $T$ by a \textit{collapse move}. In the quotient labeled graph sense, this move means we shrink an edge $e$ with label $\lambda(e)=1$ down to a vertex. The reverse of this move is called an \textit{expansion move}. These two moves correspond to the graph of groups isomorphism $A\ast_C C\cong A$. A finite sequence of these moves is called an \textit{elementary deformation}.
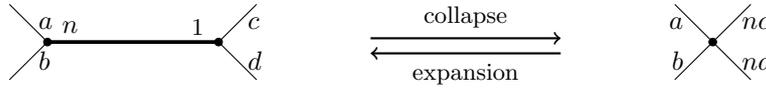
\begin{center}
\begin{tikzpicture}
      \tikzset{enclosed/.style={draw, circle, inner sep=0pt, minimum size=.1cm, fill=black}}
      
      \node[enclosed, label={right, yshift=.2cm: \small $n$}] at (2.25,3.25) {};
      \node[enclosed, label={left, yshift=.2cm: \footnotesize $1$}] at (4.5,3.25) {};
      \node[enclosed] at (11,3.25) {};

      \draw[line width=0.5mm] (2.25,3.25) -- (4.5,3.25) node[midway, sloped, above] {};
      \draw (2.25,3.25) -- (1.75, 3.75) node[midway, right] {\small $a$};
      \draw (2.25,3.25) -- (1.75, 2.75) node[midway, right] {\small $b$};
      \draw (4.5,3.25) -- (5, 3.75) node[midway, right] {\small $c$};
      \draw (4.5,3.25) -- (5, 2.75) node[midway, right] {\small $d$};

      \draw[->, thick] (6.5, 3.3) -- (9, 3.3) node[midway, above] {\footnotesize collapse};
      \draw[<-, thick] (6.5, 3.1) -- (9, 3.1) node[midway, below] {\footnotesize expansion};
      
      \draw (11,3.25) -- (10.5, 3.75) node[midway, left] {\small $a$};
      \draw (11,3.25) -- (10.5, 2.75) node[midway, left] {\small $b$}; 
      \draw (11,3.25) -- (11.5, 3.75) node[midway, right]{\small $nc$};
      \draw (11,3.25) -- (11.5, 2.75) node[midway, right]{\small $nd$};
\end{tikzpicture}
\end{center}
A $G$-tree is called \textit{reduced} if there is no collapsible edge. Given a $G$-tree $T$, the \textit{deformation space} $\mathscr{D}$ of $T$ is the set of all $G$-trees related to $T$ by an elementary deformation. Forester proved that if $G$ is a non-elementary GBS group, all such $G$-trees belong to the same deformation space \cite{M02}. We call it the \textit{canonical deformation space} of $G$.

\begin{definition}
Let $\Gamma$ be a labeled graph representing a GBS group $G$. A loop $e$ is called an \textit{ascending loop} if $\lambda(e)=\pm 1$. In addition, it is a \textit{strict ascending loop} if $\lambda(\overline{e})\neq \pm 1$. A loop $e$ is called a \textit{virtually ascending loop} if $\lambda(e)\neq \pm 1$ and $\lambda(e)|\lambda(\overline{e})$, and is a \textit{strict virtually ascending loop} if, in addition,  $\lambda(e)\neq \pm \lambda(\overline{e})$. We say $G$ is \textit{ascending} if the canonical deformation space of $G$ contains a GBS tree whose labeled graph has a strict ascending loop. Otherwise $G$ is \textit{non-ascending}.
\end{definition}

Next we will illustrate deformation moves between GBS trees. These moves have been defined and studied in \cite{M06} and \cite{MM09}.

There are two forms of \textit{slide moves}:
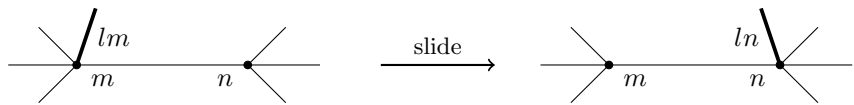
\begin{center}
\begin{tikzpicture}
      \tikzset{enclosed/.style={draw, circle, inner sep=0pt, minimum size=.1cm, fill=black}}
      
      \node[enclosed, label={right, yshift=-.2cm: \small $m$}] at (2.25,3.25) {};
      \node[enclosed, label={left, yshift=-.2cm: \small $n$}] at (4.5,3.25) {};
      \node[enclosed, label={right, yshift=-.2cm: \small $m$}] at (9.25,3.25) {};
      \node[enclosed, label={left, yshift=-.2cm: \small $n$}] at (11.5,3.25) {};

      \draw (2.25,3.25) -- (4.5,3.25) node[midway, sloped, above] {};
      \draw (2.25,3.25) -- (1.75, 3.75) node[midway, right] {};
      \draw (2.25,3.25) -- (1.35, 3.25) node[midway, right] {};
      \draw (2.25,3.25) -- (1.75, 2.75) node[midway, right] {};
      \draw[line width=0.5mm] (2.25,3.25) -- (2.5, 4) node[midway, right] {\small $lm$};
      \draw (4.5,3.25) -- (5, 3.75) node[midway, right] {};
      \draw (4.5,3.25) -- (5.45, 3.25) node[midway, right] {};
      \draw (4.5,3.25) -- (5, 2.75) node[midway, right] {};

      \draw[->, thick] (6.25, 3.25) -- (7.75, 3.25) node[midway, above] {\footnotesize slide};
      
      \draw (9.25,3.25) -- (11.5,3.25) node[midway, sloped, above] {};
      \draw (9.25,3.25) -- (8.75, 3.75) node[midway, right] {};
      \draw (9.25,3.25) -- (8.35, 3.25) node[midway, right] {};
      \draw (9.25,3.25) -- (8.75, 2.75) node[midway, right] {};
      \draw[line width=0.5mm] (11.5,3.25) -- (11.25, 4) node[midway, left] {\small $ln$};
      \draw (11.5,3.25) -- (12, 3.75) node[midway, right] {};
      \draw (11.5,3.25) -- (12.45, 3.25) node[midway, right] {};
      \draw (11.5,3.25) -- (12, 2.75) node[midway, right] {};
\end{tikzpicture}
\end{center}

and 

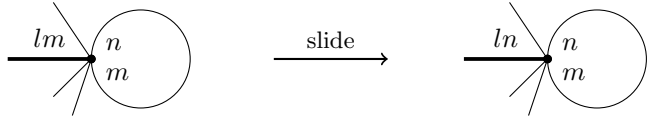
\begin{center}
\begin{tikzpicture}
      \tikzset{enclosed/.style={draw, circle, inner sep=0pt, minimum size=.1cm, fill=black}}
      
      \node[enclosed, label={right, yshift=.2cm: \small $n$}] at (2.35,3.25) {};
      \node[enclosed, label={right, yshift=-.2cm: \small $m$}] at (2.35,3.25) {};
      \node[enclosed, label={right, yshift=.2cm: \small $n$}] at (8.35,3.25) {};
      \node[enclosed, label={right, yshift=-.2cm: \small $m$}] at (8.35,3.25) {};
      \node[circle, draw, minimum size=1.3cm] at (3,3.25) {};
      \node[circle, draw, minimum size=1.3cm] at (9,3.25) {};

      \draw (2.35,3.25) -- (1.85, 4) node[midway, right] {};
      \draw[line width=0.5mm]  (2.35,3.25) -- (1.25, 3.25) node[midway, above] {\small{$lm$}};
      \draw (2.35,3.25) -- (1.85, 2.75) node[midway, right] {};
      \draw (2.35,3.25) -- (2.1, 2.5) node[midway, right] {};

      \draw[->, thick] (4.75, 3.25) -- (6.25, 3.25) node[midway, above] {\footnotesize slide};
            
      \draw (8.35,3.25) -- (7.85, 4) node[midway, right] {};
      \draw[line width=0.5mm]  (8.35,3.25) -- (7.25, 3.25) node[midway, above] {\small{$ln$}};
      \draw (8.35,3.25) -- (7.85, 2.75) node[midway, right] {};
      \draw (8.35,3.25) -- (8.1, 2.5) node[midway, right] {};
\end{tikzpicture}
\end{center}

An \textit{induction move} is defined as follows:

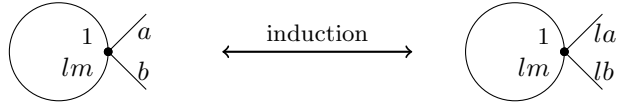
\begin{center}
\begin{tikzpicture}
      \tikzset{enclosed/.style={draw, circle, inner sep=0pt, minimum size=.1cm, fill=black}}
      
      \node[enclosed, label={left, yshift=.2cm: \footnotesize $1$}] at (2,3.25) {};
      \node[enclosed, label={left, yshift=-.2cm: \small $lm$}] at (2,3.25) {};
      \node[enclosed, label={left, yshift=.2cm: \footnotesize $1$}] at (8,3.25) {};
      \node[enclosed, label={left, yshift=-.2cm: \small $lm$}] at (8,3.25) {};
      \node[circle, draw, minimum size=1.3cm] at (1.35,3.25) {};
      \node[circle, draw, minimum size=1.3cm] at (7.35,3.25) {};

      \draw (2,3.25) -- (2.5, 3.75) node[midway, right] {\small $a$};
      \draw (2,3.25) -- (2.5, 2.75) node[midway, right] {\small $b$};

      \draw[<->, thick] (3.5, 3.25) -- (6, 3.25) node[midway, above] {\footnotesize induction};
            
      \draw (8,3.25) -- (8.5, 3.75) node[midway, right] {\small $la$};
      \draw (8,3.25) -- (8.5, 2.75) node[midway, right] {\small $lb$};
\end{tikzpicture}
\end{center}

An \textit{$\mathscr{A}^{\pm 1}$-move} is defined as follows:

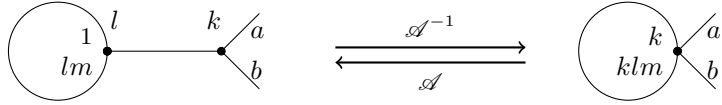
\begin{center}
\begin{tikzpicture}
      \tikzset{enclosed/.style={draw, circle, inner sep=0pt, minimum size=.1cm, fill=black}}
      
      \node[enclosed, label={left, yshift=.2cm: \footnotesize $1$}] at (0.5,3.25) {};
      \node[enclosed, label={left, yshift=-.2cm: \small $lm$}] at (0.5,3.25) {};
      \node[enclosed, label={above, xshift=0.1cm, yshift=.1cm: \small $l$}] at (0.5,3.25) {};
      \node[enclosed, label={above, xshift=-0.1cm, yshift=.1cm: \small $k$}] at (2,3.25) {};
      \node[enclosed, label={left, yshift=.2cm: \small $k$}] at (8,3.25) {};
      \node[enclosed, label={left, yshift=-.2cm: \small $klm$}] at (8,3.25) {};
      \node[circle, draw, minimum size=1.3cm] at (-0.15,3.25) {};
      \node[circle, draw, minimum size=1.3cm] at (7.35,3.25) {};

      \draw (2,3.25) -- (2.5, 3.75) node[midway, right] {\small $a$};
      \draw (2,3.25) -- (2.5, 2.75) node[midway, right] {\small $b$};
      \draw (0.5,3.25) -- (2, 3.25) node[midway, right] {};

      \draw[->, thick] (3.5, 3.3) -- (6, 3.3) node[midway, above] {\footnotesize $\mathscr{A}^{-1}$};
      \draw[<-, thick] (3.5, 3.1) -- (6, 3.1) node[midway, below] {\footnotesize $\mathscr{A}$};
            
      \draw (8,3.25) -- (8.5, 3.75) node[midway, right] {\small $a$};
      \draw (8,3.25) -- (8.5, 2.75) node[midway, right] {\small $b$};
\end{tikzpicture}
\end{center}
In the \textit{$\mathscr{A}^{\pm 1}$-move} we require that $k,l\neq \pm 1$ and the left vertex has no other edges incident to it.

Now we state the main theorem in \cite{MM09} which we will use in the latter of this paper.

\begin{theorem}\cite[Theorem 1.1]{MM09}\label{theorem: 3 moves}
In a deformation space of cocompact $G$-trees, any two reduced trees are related by a finite sequence of slides, inductions, and $\mathscr{A}^{\pm 1}$-moves, with all intermediate trees reduced.
\end{theorem}

From this theorem, we can deduce the following corollary:
\begin{cor}\label{coro: slide sequence}
In a non-ascending deformation space of cocompact $G$-trees, any two reduced trees are related by a finite sequence of slide moves, with all intermediate trees reduced.
\end{cor}

\subsection{The modular homomorphism}
Let $G$ be a non-elementary GBS group with quotient labeled graph $\Gamma$. We define the \textit{modular homomorphism} as the composition $q:G\rightarrow \pi_1^{top}(\Gamma)\rightarrow \mathbb{Q}^{\times}$ where the first map is given by killing the normal closure of elliptic elements and the second map is given by
$$(e_1,...,e_s)\mapsto \prod_{i=1}^s \frac{\lambda(\overline{e}_i)}{\lambda(e_i)}$$
We call $q(g)$ the \textit{modulus} of $g$,
We will mostly use the second map in the latter of this paper. We remark that the second map does not depend on the choice of labeled graph (see Remark 6.5 in \cite{M06}). We say $G$ is \textit{unimodular} if $q(G)\subseteq \{1,-1\}$. 
There are several equivalent definitions of the modular homomorphism; details can be found in \cite{BK90} \cite{M06} and \cite{MM08}.

\subsection{Monotone cycles and mobile edges}

\begin{definition}\cite[Definition 3.2]{MM08}
Let $\Gamma$ be a labeled graph representing $G$ and $e\in E(\Gamma)$. An edge path $(e_1,...,e_s)$ is an \textit{$e$-edge path} if the following holds:
\begin{itemize}[topsep=0pt, itemsep=1ex]
     \item[(a)] $e_i\neq e,\overline{e}$ for $i=1,...,s$;
     \item[(b)] $o(e)=o(e_1)$; and
     \item[(c)] $\lambda(e_i)|\lambda(e)q(e_1,...,e_{i-1})$ for $i=1,...,s$.
\end{itemize}
An $e$-edge path is an \textit{$e$-integer cycle} if, in addition we have:
\begin{itemize}[topsep=0pt, itemsep=1ex]
     \item[(d)] $o(e_1)=t(e_s)$; and
     \item[(e)] $q(e_1,...,e_s)\in \mathbb{Z}$.
\end{itemize}
\end{definition}

We say an $e$-edge path or $e$-integer cycle is \textit{strict} if $|q(e_1,...,e_s)|\neq 1$. Note that an $e$-edge path allows us to slide $e$ along $(e_1,...,e_s)$ and the label of $e$ changes from $\lambda(e)$ to $\lambda(e)q(e_1,...,e_s)$. 

\begin{definition}\cite[Definition 3.4]{MM08}
An edge path $(e_1,...,e_s,e)$ is a \textit{monotone cycle} if $(e_1,...,e_s)$ is an $\overline{e}$-edge path and $q(e_1,...,e_s,e)\in \mathbb{Z}$. A monotone cycle is \textit{strict} if the modulus is not equal to $\pm 1$.

Note that if $s=0$, then the (strict) monotone cycle is either a (strict) ascending loop or a (strict) virtually ascending loop. More generally, when $\Gamma$ has a monotone cycle $\gamma$ with final edge $e$, sliding $\overline{e}$ along the previous edges in the cycle results in a new labeled graph $\Gamma'$ in which $e$ is either a ascending loop or virtually ascending loop with modulus $q(\gamma)$.
\end{definition}

\begin{definition}
Let $G$ be a GBS group and $\mathscr{D}$ the canonical deformation space of $G$.  We denoted by RLG($G$) the set of all reduced labeled graphs in $\mathscr{D}$. In general, RLG($G$) contains infinitely many reduced labeled graphs.
\end{definition}

\begin{definition}\cite[Definition 3.12]{MM08}\label{def: mobile edge}
Let $\Gamma\in$ RLG($G$). An edge $e\in E(\Gamma)$ is \textit{mobile} if either:
\begin{itemize}[topsep=0pt, itemsep=1ex]
     \item[(a)] there is a strict monotone cycle of the form $(e_1,...,e_s,e)$ or $(e_1,...,e_s,\overline{e})$; or
     \item[(b)] $\Gamma$ contains a strict $e$-integer cycle or a strict $\overline{e}$-integer cycle.
\end{itemize}
An edge that is not mobile is called \textit{non-mobile}. Note that mobility is a property of geometric edges: $e$ is mobile if and only if $\overline{e}$ is mobile. In addition, the mobility of an edge is preserved by slide moves; see Corollary 3.15 in \cite{MM08}. We remark that there is an algorithm to determine whether a given edge $e\in E(\Gamma)$ is mobile or not; see Remark 3.13 in \cite{MM08}.
\end{definition}

\begin{remark}\cite[Lemma 3.6]{MM08}\label{remark: non-sacending}
If $\Gamma$ has a strict monotone cycle, then $G$ is ascending. This is because if there exists a strict monotone cycle $\gamma=(e_1,...,e_s,e)$, then we can slide $\overline{e}$ along the path $(e_1,...,e_s)$ turning $e$ into either a strict virtually ascending loop or a strict ascending loop. Therefore, after an $\mathscr{A}$-move in the first case we have $G$ is ascending, which is a contradiction. Hence, for every mobile edge $e$ in a non-ascending deformation space, there exists either a strict $e$-integer cycle or a strict $\overline{e}$-integer cycle.
\end{remark}

\section{When a GBS group is ascending}

\subsection{Existence of strict monotone cycles} Let $G$ be a non-elementary GBS group and $\Gamma\in$ RLG($G$). Suppose $e$ is a mobile edge in $\Gamma$, it is unknown whether there exists a strict monotone cycle with the last edge $e$ or $\overline{e}$. In this subsection, we will work out some methods to find strict monotone cycles in $\Gamma$.

\begin{definition}\cite[Definition 3.9]{MM08}
Given $\Gamma\in$ RLG($G$) and $e\in E(\Gamma)$, we denote by $\mbox{S}(\Gamma,e)\subseteq$ RLG($G$) the set of reduced labeled graphs obtained from $\Gamma$ by a sequence of slides of $e$ and $\overline{e}$. $\mbox{S}(\Gamma,e)$ is then called the \textit{slide space of e based at $\Gamma$}.
\end{definition}

\begin{prop}\label{prop: smc algorithm }
Let $\Gamma$ be a $n$-rose graph. Then there is an algorithm to determine whether $\Gamma$ has a strict monotone cycle or not.
\end{prop}
\begin{proof}
We first consider the case $n=1$. Then the first Betti number of $\Gamma$ is 1. By Remark 4.1 in \cite{MM08}, we can determine algorithmically whether $\Gamma$ contains a strict monotone cycle. 

Now we provide a proof for the general case $n\geq 2$. In this case, $\Gamma$ is a wedge sum of $n$ oriented loops $g_1,...,g_n$. First, we can check whether the set $\{q(g_1),q(\overline{g}_1),..., q(g_n),q(\overline{g}_n)\}$ contains a non-trivial integer. An affirmative answer implies that $\Gamma$ contains either a strict ascending loop or a strict virtually ascending loop, which is a strict monotone cycle. Otherwise, given a mobile edge $e$, we notice that $e\in \{g_1,\overline{g}_1,...,g_n,\overline{g}_n\}$. Without loss of generality, we may assume that $e=g_1$. If $\Gamma$ has a strict monotone cycle $\gamma=(e_1,...,e_s,e)$, we have $e_t\in \{g_2,\overline{g}_2,...,g_n,\overline{g}_n\}$ for all $t=1,...,s$. By definition, there exists a $(n-1)$-tuple of integers $(x_2,...,x_n)$ such that $\lambda(e)|\lambda(\overline{e})\prod_{h=2}^nq(g_h)^{x_h}$ and $\lambda(e)\neq \pm \lambda(\overline{e})\prod_{h=2}^nq(g_h)^{x_h}$.

We now denote all different prime factors of $q(g_1),..., q(g_n)$ by $p_1,...,p_r$. Then every label in $\Gamma$ has the form $p_1^{\epsilon_1}\cdot\cdot\cdot p_r^{\epsilon_r}$ with each $\epsilon_i\in \mathbb{Z}$. Therefore we can associate a $r$-tuple $(\epsilon_1,...,\epsilon_r)$ for every label in $\Gamma$. We then denoted by $(\sigma_1,...,\sigma_r)$ (resp. $(\sigma_1',...,\sigma_r')$) the $r$-tuple of $\lambda(e)$ (resp. $\lambda(\overline{e})$), and we denoted by $(\alpha^h_1,...,\alpha^h_r)$ the $r$-tuple of $q(g_h)$ for $h=2,...,n$. Then the $(n-1)$-tuple of integers $(x_2,...,x_n)$ above gives us an integer solution of the system of inequalities $\sigma_l\leq \sigma_l'+\sum_{h=2}^n x_h\alpha_l^h$ for all $l=1,...,r$. In addition, there must exist one $l$ such that $\sigma_l< \sigma_l'+\sum_{h=2}^n x_h\alpha_l^h$ because we require that $\lambda(e)\neq \lambda(\overline{e})\prod_{h=2}^nq(g_h)^{x_h}$. We then define 
$$\Pi(e)=\Bigl\{(x_2,...,x_n)\ |\  \sigma_l\leq \sigma_l'+\sum_{h=2}^n x_h\alpha_l^h\ \mbox{for all}\  l=1,...,r \ \mbox{except one with strict less than} \Bigl\}$$
Thus each element in $\Pi(e)$ corresponds to a candidate of strict monotone cycle with the last edge $e$. In addition, each element in $\Pi(e)$  gives us a candidate of label of $\overline{e}$ after sliding over a strict monotone cycle with the last edge $e$, i.e. $\lambda(\overline{e})\prod_{h=2}^nq(g_h)^{x_h}$.

We notice that, by Lemma 4, Lemma 5 and Theorem 3 in \cite{FA17}, we can find a set of $(n-1)$-tuple of integers
$$\Lambda(e)=\bigcup_{i=1}^L\Bigl\{(y_2,...,y_n)\ |\  w^i_l\leq \sigma_l'+\sum_{h=2}^n y_h\alpha_l^h\ \mbox{for all}\ l=1,...,r\Bigl\}$$
such that for each element $(y_2,...,y_n)$, the number $\lambda(\overline{e})\prod_{h=2}^nq(g_h)^{y_h}$ is a possible label that can appear on $\overline{e}$ after a slide move of $\overline{e}$ over an edge cycle, where $L$ is finite and $w^i_l$ is a computable non-negative integer for each $i=1,...,L$. Since there is only one vertex in $\Gamma$, then every slide move of $\overline{e}$ is over a cycle in $\Gamma$. It follows that $\Lambda(e)$ corresponds to a set of all possible labels of $\overline{e}$ as a consequence of slide moves of $\overline{e}$. Thus we can determine whether $\Gamma$ contains a strict monotone cycle with the last edge $e$ by checking that whether the set $\Pi(e)\cap \Lambda(e)$ is empty. We note that the set $\Pi(e)\cap \Lambda(e)$ corresponds to lattice points in a polyhedron (possibly unbounded) in $\mathbb{R}^{n-1}$. Therefore it is suffices to detect whether there exists a lattice point in such polyhedron. This is the integer linear programming problem and it can be done by the algorithm in section 4 in \cite{L83}.

Finally, by repeating the above process for $\overline{e}$ and other mobile edges in $\Gamma$, we can determine whether $\Gamma$ contains strict monotone cycle or not.
\end{proof}

\subsection{Slide moves of strict monotone cycles}
Given a non-elementary GBS group $G$ and $\Gamma\in$ RLG($G$), suppose $\Gamma$ has a strict monotone cycle $\gamma$. We denote the last edge of $\gamma$ by $e$. By Corollary 3.15 in \cite{MM08}, we know that slide moves cannot change the mobility of the edge $e$. However it is unknown when slide moves will preserve strict monotone cycles. In this subsection, we will study slide moves of edges in $E(\gamma)$ to see how they affect $\gamma$. To simplify our discussion, let $G^{+}$ be the GBS group represented by the labeled graph $(\Gamma, |\lambda|)$ and we only consider slide moves in $\mbox{RLG}$($G^{+}$).

\begin{notation}
Given $\Gamma\in$ RLG($G$), $e\in E(\Gamma)$ and $A$ an $e$-edge path, we will use the notation $e/A$ to denote the slide move of $e$ over $A$.
\end{notation}

\begin{definition}
Given $\Gamma\in$ RLG($G$) and $e\in E(\Gamma)$, suppose $A=(e_1,...,e_s)$ is an $e$-edge path in $\Gamma$. We call $\sigma=(e_i,...,e_j)$ a \textit{redundant subcycle} of $A$ if $o(e_i)=t(e_j)$ and $q(\sigma)=1$ for some $1\leq i\leq j\leq s$.  
\end{definition}

\begin{remark}\label{remark: redundant cycles}
Note that if $A$ has a redundant subcycle $\sigma=(e_i,...,e_j)$, one can check the complement $A\setminus \sigma=(e_1,...,e_{i-1},e_{j+1},...,e_s)$ is also an $e$-edge path. In addition, we have $e/A=e/(A\setminus \sigma)$, i.e. $e/A$ and $e/(A\setminus \sigma)$ will result in same labeled graph. Therefore, whenever we speak of an $e$-edge path $A$, we will always refer to the path that has no redundant subcycles. Moreover, whenever we speak of a strict monotone cycle $\gamma=(e_1,...,e_s,e)$, we will always assume that the subpath $(e_1,...,e_s)$ has no redundant subcycles.
\end{remark}

\begin{lemma}\label{lemma: slide over non-mobile edge}
Suppose $\Gamma\in$ RLG($G^+$) and $\gamma=(e_1,...,e_s,e)$ is a strict monotone cycle in $\Gamma$ for some $s\in \mathbb{Z}_{\geq 0}$. Further, suppose there is a $\Gamma'\in$ RLG($G^+$) such that $\Gamma'$ can be obtained from $\Gamma$ by either $e/f,\overline{e}/f,e_i/f$ or $\overline{e}_i/f$ where $f\not\in\{e,\overline{e}\} $. Then $\Gamma'$ contains a strict monotone cycle.
\end{lemma}
\begin{proof}
For $e/f$, since $e$ can slide over $f$ in $\Gamma$, then $\lambda(f)|\lambda(e)$. In addition, we have $\lambda'(e)=\lambda(e)q(f)$. Since $\gamma$ is a strict monotone cycle in $\Gamma$, then $\overline{e}$ can slide over $(e_1,...,e_s)$ and the resulting label is $\lambda(\overline{e})q(e_1,...,e_s)$. Moreover, we have $\lambda(e)|\lambda(\overline{e})q(e_1,...,e_s)$. Thus we have $\lambda(f)|\lambda(\overline{e})q(e_1,...,e_s)$. This means that $\overline{e}$ can slide over the edge path $(e_1,...,e_s,f)$ in $\Gamma'$ and the resulting label of $\overline{e}$ after sliding over it is $\lambda(\overline{e})q(e_1,...,e_s)q(f)$. In addition, we have $\lambda(e)q(f)|\lambda(\overline{e})q(e_1,...,e_s)q(f)$. This implies that $(e_1,...,e_s,f,e)$ is a strict monotone cycle in $\Gamma'$.

For $\overline{e}/f$, since $\overline{e}$ can slide over $f$ in $\Gamma$, then the new $\overline{e}$ in $\Gamma'$ can slide over $\overline{f}$ back to its original position, and the resulting label is the same as the original one, which is $\lambda(\overline{e})$. Thus the edge cycle $(\overline{f},e_1,...,e_s,e)$ in $\Gamma'$ is a strict monotone cycle.

For $e_i/f$, we first notice that $\gamma$ need not be an embedded edge cycle, so $e_i$ may appear more than once in $\gamma$. To reduce subscript use, we denote the edge $e_i$ by $g$, so we can rewrite $\gamma$ as
$$\gamma=(e_1,...,e_{i_1-1},e_{i_1},e_{i_1+1}...,e_{i_k-1},e_{i_k},e_{i_k+1},...,e_s,e)$$
where each $e_{i_j}\in \{g,\overline{g}\}$ and $e_r\not\in \{g,\overline{g}\}$ for all $j=1,...,k$ and $r\neq i_j$. Now we will establish a \enquote{surgery} to find a strict monotone cycle in $\Gamma'$. The procedure is as follows. For the strict monotone cycle $\gamma$,
\begin{itemize}[topsep=0pt, itemsep=1ex]
\item[-] If $e_{i_j}=g$, then replace $e_{i_j}$ in $\gamma$ by the edge path $(f,e_{i_j})$.
\item[-] If $e_{i_j}=\overline{g}$, then replace $e_{i_j}$ in $\gamma$ by the edge path $(e_{i_j},\overline{f})$.
\end{itemize}
We denoted by $\gamma'$ the resulting edge cycle obtained from $\gamma$ by the above surgery. One can check that $\gamma'$ is a strict monotone cycle in $\Gamma'$.

For $\overline{e}_i/f$, similarly, we will provide another surgery of $\gamma$ to find a strict monotone cycle in $\Gamma'$. For the strict monotone cycle $\gamma$, recall that $g=e_i$, so
\begin{itemize}[topsep=0pt, itemsep=1ex]
\item[-] If $e_{i_j}=g$, then replace $e_{i_j}$ in $\gamma$ by the edge path $(e_{i_j},\overline{f})$.
\item[-] If $e_{i_j}=\overline{g}$, then replace $e_{i_j}$ in $\gamma$ by the edge path $(f,e_{i_j})$.
\end{itemize}
One can check the resulting edge cycle $\gamma'$ after the above surgery of $\gamma$ is a strict monotone cycle in $\Gamma'$.
\end{proof}

\begin{lemma}\label{lemma: n rose slide over mobile edge}
Suppose $\Gamma\in$ RLG($G^+$) is a $n$-rose graph and $\gamma=(e_1,...,e_s,e)$ is a strict monotone cycle in $\Gamma$ for some $s\in \mathbb{Z}_{\geq0}$. Further, suppose there is a $\Gamma'\in$ RLG($G^+$) such that $\Gamma'$ can be obtained from $\Gamma$ by either $e_i/e$ or $\overline{e}_i/e$, then $\Gamma'$ contains a strict monotone cycle.
\end{lemma}
\begin{proof}
First, if $s=0$ or if $n=1$, then there is no edge $e_i$ to slide over $e$. Therefore $\Gamma=\Gamma'$ and $e$ is a strict monotone cycle of one single edge in $\Gamma'$. 

Next we consider the general case where $n>1$, then $\Gamma$ is a wedge sum of $n$ oriented loops. We denote $e_i$ by $f$. We first notice that $\gamma$ need not be an embedded edge cycle, so $f$ may appear more than once in $\gamma$. Thus we can rewrite $\gamma$ as
$$\gamma=(e_1,...,e_{i_1-1},e_{i_1},e_{i_1+1},...,e_{i_k-1},e_{i_k},e_{i_k+1},...,e_s,e)$$
where each $e_{i_j}\in \{f,\overline{f}\}$ and $e_r\not\in \{f,\overline{f}\}$ for all $j=1,...,k$ and $r\neq i_j$. We remark that if $n=2$, then there is no such $e_r$, i.e. $\gamma=(f,...,f,e)$ or $(\overline{f},...,\overline{f},e)$ by Remark \ref{remark: redundant cycles}.  In addition, since we are considering the case $e_i=f$, so there must exist a $j$ such that $e_{i_j}=f$.  For simplicity, we will use $\lambda'$ to denote labels in $\Gamma'$.

For $e_i/e$, we have $\lambda(e)|\lambda(f)$. Suppose that $e_{i_1}=f$. Since $\gamma$ is a strict monotone cycle, then $\lambda(e_p)|\lambda(\overline{e})q(e_1,...,e_{p-1})$ for $p=1,...,s$. In particular, we have $\lambda(f)|\lambda(\overline{e}) q(e_1,...,e_{i_1-1})$. It follows that $\lambda(e)|\lambda(\overline{e})q(e_1,...,e_{i_1-1})$. Then we have the following cases:
\begin{itemize}[topsep=0pt, itemsep=1ex]
\item[-] If $\lambda(e)$ properly divides $\lambda(\overline{e})q(e_1,...,e_{i_1-1})$, then we have $(e_1,...,e_{i_1-1},e)$ is a strict monotone cycle in $\Gamma'$.
\item[-] Otherwise $\lambda(e)=\lambda(f)=\lambda(\overline{e})q(e_1,...,e_{i_1-1})$. Since $\lambda(e_p)|\lambda(\overline{e})q(e_1,...,e_{p-1})$ for $p=1,...,s$, then we have $$\lambda(e_p)|\lambda(\overline{e})q(e_1,...,e_{i_1-1})\frac{\lambda(\overline{f})}{\lambda(f)}q(e_{i_1+1},...,e_{p-1})\ \Rightarrow \lambda(e_p)|\lambda(\overline{f})q(e_{i_1+1},...,e_{p-1})$$ for $p=i_1+1,...,s$.
\begin{itemize}[topsep=0pt, itemsep=1ex]
\item[-] If $k=1$, we claim that $(e_{i_1+1},...,e_s,e,f)$ is a strict monotone cycle in $\Gamma'$. We notice that the above division implies that $(e_{i_1+1},...,e_s)$ is a $\overline{f}$-edge path in $\Gamma'$. Moreover, the resulting label of $\overline{f}$ after sliding over it is $\lambda(\overline{f})q(e_{i_1+1},...,e_s)=\lambda(\overline{e})q(e_1,...,e_s)$. Recall that $\gamma$ is a strict monotone cycle, so $\lambda(e)$ properly divides $\lambda(\overline{e})q(e_1,...,e_s)$. This implies that $(e_{i_1+1},...,e_s,e)$ is a $\overline{f}$-edge path in $\Gamma'$ and the resulting label of $\overline{f}$ after sliding over it is $\lambda(\overline{e})q(\gamma)$. Since $\lambda'(f)= \frac{\lambda(f)}{\lambda(e)}\lambda(\overline{e})=\lambda(\overline{e})$, then we have $\lambda'(f)|q(\gamma)\lambda(\overline{e})$. This shows that $(e_{i_1+1},...,e_s,e,f)$ is a strict monotone cycle in $\Gamma'$.
\item[-] If $k>1$, then there exists $e_{i_2}\in \{f,\overline{f}\}$ in the edge subpath $(e_{i_1+1},...,e_s)$.

If $e_{i_2}=f$, we claim that $(e_{i_1+1},...,e_{i_2-1},e,f)$ is a strict monotone cycle in $\Gamma'$. We notice that the above division implies $(e_{i_1+1},...,e_{i_2-1})$ is a $\overline{f}$-edge path in $\Gamma'$. Moreover, the resulting label of $\overline{f}$ after sliding over it is $\lambda(\overline{f})q(e_{i_1+1},...,e_{i_2-1})$. Since $\gamma$ is a strict monotone cycle, then we have $\lambda(f)|\lambda(\overline{f})q(e_{i_1+1},...,e_{i_2-1})$, i.e. $\lambda(f)h=\lambda(\overline{f})q(e_{i_1+1},...,e_{i_2-1})$. By Remark \ref{remark: redundant cycles}, we have $h\neq 1$, otherwise $(e_{i_1+1},...,e_{i_2})$ is a redundant subcycle of $\gamma$. Since $\lambda(e)=\lambda(f)$, we have $\lambda(e)h=\lambda(\overline{f})q(e_{i_1+1},...,e_{i_2-1})$. This implies that $(e_{i_1+1},...,e_{i_2-1},e)$ is a $\overline{f}$-path and the resulting label of $\overline{f}$ after sliding over it is $\lambda(\overline{e})h$. Since $\lambda'(f)=\lambda(\overline{e})$, we have $\lambda'(f)|\lambda(\overline{e})h$. This shows that $(e_{i_1+1},...,e_{i_2-1},e,f)$ is a strict monotone cycle in $\Gamma'$.

On the other hand, if $e_{i_2}=\overline{f}$, we claim $(f,e_{i_1+1},...,e_{i_2-1},\overline{f},e_1,...,e_{i_1-1},e)$ is a strict monotone cycle in $\Gamma'$. We notice that the above division implies $\lambda(\overline{f})|\lambda(\overline{f})q(e_{i_1+1},...,e_{i_2-1})$, i.e. $q(e_{i_1+1},...,e_{i_2-1})\in \mathbb{Z}$. By Remark \ref{remark: redundant cycles}, we have $q(e_{i_1+1},...,e_{i_2-1})\neq 1$. This can not happen if $n=2$. Thus it suffices to consider $n>2$. We first notice that $\lambda'(\overline{e})=\lambda(\overline{e})=\lambda'(f)$, so $\overline{e}$ can slide over $f$ in $\Gamma'$ and the resulting label of $\overline{e}$ is $\lambda(\overline{f})$. By the above division, we have $(f,e_{i_1+1},...,e_{i_2-1})$ is an $\overline{e}$-edge path and the resulting label of $\overline{e}$ after sliding over it is $\lambda(\overline{f})q(e_{i_1+1},...,e_{i_2-1})$. Since $\lambda(\overline{f})|\lambda(\overline{f})q(e_{i_1+1},...,e_{i_2-1})$, then $(f,e_{i_1+1},...,e_{i_2-1},\overline{f})$ is an $\overline{e}$-edge path and the resulting label of $\overline{e}$ is $\lambda(\overline{e})q(e_{i_1+1},...,e_{i_2-1})$. Recall that $\lambda(e_p)|\lambda(\overline{e})q(e_1,...,e_{p-1})$ for $p=1,...,s$, then $\lambda(e_p)|(\lambda(\overline{e})q(e_{i_1+1},...,e_{i_2-1}))q(e_1,...,e_{p-1})$ for $p=1,...,i_1-1$. Therefore $(f,e_{i_1+1},...,e_{i_2-1},\overline{f},e_1,...,e_{i_1-1})$ is an $\overline{e}$-edge path and the resulting label of $\overline{e}$ is $\lambda(\overline{e})q(e_{i_1+1},...,e_{i_2-1})q(e_1,...,e_{i_1-1})=\lambda(e)q(e_{i_1+1},...,e_{i_2-1})$. Since $\lambda'(e)=\lambda(e)$ and $\lambda(e)$ properly divides $\lambda(e)q(e_{i_1+1},...,e_{i_2-1})$, then we have $(f,e_{i_1+1},...,e_{i_2-1},\overline{f},e_1,...,e_{i_1-1},e)$ is a strict monotone cycle in $\Gamma'$.
\end{itemize}
\end{itemize}
On the other hand, suppose that $e_{i_1}=\overline{f}$. We first notice that, by Remark \ref{remark: redundant cycles}, this subcase can not happen if $n=2$. So it suffices to consider $n>2$. Since $\gamma$ is a strict monotone cycle, we have $\lambda(\overline{f})|\lambda(\overline{e})q(e_1,...,e_{i_1-1})$. Recall that $\lambda'(f)=\frac{\lambda(f)}{\lambda(e)}\lambda(\overline{e})$. This implies that $\lambda(\overline{f})|\lambda'(f)q(e_1,...,e_{i_1-1})$. Since there exists $j$ such that $e_{i_j}=f$, then we must have $k>1$, i.e. there exists $e_{i_2}\in \{f,\overline{f}\}$ in the edge subpath $(e_{i_1+1},...,e_s)$.
\begin{itemize}[topsep=0pt, itemsep=1ex]
\item[-] If $\lambda(\overline{f})$ properly divides $\lambda'(f)q(e_1,...,e_{i_1-1})$, then we have $(e_1,...,e_{i_1-1},\overline{f})$ is a strict monotone cycle in $\Gamma'$.
\item[-] Otherwise we have $\lambda(\overline{f})=\lambda'(f)q(e_1,...,e_{i_1-1})$. This implies that $\lambda(f)=\lambda(e)$, $\lambda'(f)=\lambda(\overline{e})$ and $\lambda(\overline{f})=\lambda(\overline{e})q(e_1,...,e_{i_1-1})$. In addition, we have $$\lambda(\overline{e})q(e_1,...,e_{i_1})=\lambda(\overline{e})q(e_1,...,e_{i_1-1})\frac{\lambda(f)}{\lambda(\overline{f})}=\lambda(f)$$
If $e_{i_2}=f$, we claim that $(\overline{e},e_{i_1+1},...,e_{i_2-1},e,e_1,...,e_{i_1-1},\overline{f})$ is a strict monotone cycle in $\Gamma'$. Since $\gamma$ is a strict monotone cycle, then $\lambda(f)|\lambda(\overline{e})q(e_1,...,e_{i_2-1})$. It follows that $\lambda(f)|\lambda(f)q(e_{i_1+1},...,e_{i_2-1})$, i.e. $q(e_{i_1+1},...,e_{i_2-1})\in \mathbb{Z}$. By Remark \ref{remark: redundant cycles}, we have $q(e_{i_1+1},...,e_{i_2-1})\neq 1$, otherwise the edge cycle $(e_{i_1+1},...,e_{i_2-1})$ is a redundant subcycle of $\gamma$.  We notice that $\lambda'(\overline{e})=\lambda'(f)=\lambda(\overline{e})$, so $f$ can slide over $\overline{e}$ in $\Gamma'$ and the resulting label is $\lambda(e)=\lambda(f)$. Since $\lambda(e_p)|\lambda(\overline{e})q(e_1,...,e_{p-1})$ for $p=1,...,s$, then $\lambda(e_p)|\lambda(f)q(e_{i_1+1},...,e_{p-1})$ for $p=i_1+1,...,i_2$. This implies that $(\overline{e},e_{i_1+1},...,e_{i_2-1})$ is a $f$-edge path and the resulting label of $f$ after sliding over it is $\lambda(f)q(e_{i_1+1},...,e_{i_2-1})$. Since $\lambda(e)|\lambda(f)q(e_{i_1+1},...,e_{i_2-1})$, then $(\overline{e},e_{i_1+1},...,e_{i_2-1},e)$ is a $f$-edge path and the resulting label of $f$ after sliding over it is $\lambda(\overline{e})q(e_{i_1+1},...,e_{i_2-1})$. Recall that $\lambda(e_p)|\lambda(\overline{e})q(e_1,...,e_{p-1})$ for $p=1,...,s$, then $\lambda(e_p)|(\lambda(\overline{e})q(e_{i_1+1},...,e_{i_2-1}))q(e_1,...,e_{p-1})$ for $p=1,...,i_1-1$. Therefore $(\overline{e},e_{i_1+1},...,e_{i_2-1},e,e_1,...,e_{i_1-1})$ is a $f$-edge path and the resulting label of $f$ after sliding over it is $\lambda(\overline{e})q(e_{i_1+1},...,e_{i_2-1})q(e_1,...,e_{i_1-1})=\lambda(\overline{f})q(e_{i_1+1},...,e_{i_2-1})$. Since $\lambda'(\overline{f})=\lambda(\overline{f})$ and $\lambda(\overline{f})$ properly divides $\lambda(\overline{f})q(e_{i_1+1},...,e_{i_2-1})$, then we have $(\overline{e},e_{i_1+1},...,e_{i_2-1},e,e_1,...,e_{i_1-1},\overline{f})$ is a strict monotone cycle in $\Gamma'$.\\
On the other hand, if $e_{i_2}=\overline{f}$, we have $\lambda(\overline{f})|\lambda(f)q(e_{i_1+1},...,e_{i_2-1})$, i.e. $\lambda(\overline{f})h=\lambda(f)q(e_{i_1+1},...,e_{i_2-1})$. By Remark \ref{remark: redundant cycles}, we have $h\neq 1$, otherwise $(e_{i_1+1},...,e_{i_2})$ is a redundant cycle of $\gamma$. This shows that $(\overline{e},e_{i_1+1},...,e_{i_2-1},\overline{f})$ is a strict monotone cycle in $\Gamma'$.
\end{itemize}

For $\overline{e}_i/e$, we first notice that if there exists $e_{i_j}=\overline{f}$, then $\Gamma'$ is obtained from $\Gamma$ by $e_{i_j}/e$. By the previous argument, we know that $\Gamma'$ contains a strict monotone cycle. So it suffices to consider $e_{i_j}=f$ for all $j=1,...,k$. Since $\overline{e}_i$ can slide over $e$, we have $\lambda'(\overline{f})=\frac{\lambda(\overline{f})}{\lambda(e)}\lambda(\overline{e})$. Since $\gamma$ is a strict monotone cycle, we have $\lambda(f)|\lambda(\overline{e})q(e_1,...,e_{i_1-1})$. This implies that $\lambda(f)|\lambda'(\overline{f})q(e_1,...,e_{i_1-1})$. Then we have the following cases:
\begin{itemize}[topsep=0pt, itemsep=1ex]
\item[-] If $\lambda(f)$ properly divides $\lambda'(\overline{f})q(e_1,...,e_{i_1-1})$, then we have $(e_1,...,e_{i_1-1},f)$ is a strict monotone cycle in $\Gamma'$.
\item[-] If $\lambda(f)=\lambda'(\overline{f})q(e_1,...,e_{i_1-1})$. This case is similar to the case $e_i/e$. One can check that
\begin{itemize}[topsep=0pt, itemsep=1ex]
\item[-] If $k=1$, then $(\overline{e},e_{i_1+1},...,e_s,e,e_1,...,e_{i_1-1},f)$ is a strict monotone cycle in $\Gamma'$.
\item[-] If $k>1$, then $(\overline{e},e_{i_1+1},...,e_{i_2-1},f)$ is a strict monotone cycle in $\Gamma'$.
\end{itemize}
\end{itemize}

\end{proof}

\begin{lemma}\label{lemma: 3 rose slide over mobile edge}
Suppose $\Gamma\in$ RLG($G^+$) is a $n$-rose graph with $n\leq 3$ and $\gamma=(e_1,...,e_s,e)$ is a strict monotone cycle in $\Gamma$ for some $s\in \mathbb{Z}_{\geq0}$. Further, suppose there is a $\Gamma'\in$ RLG($G^+$) such that $\Gamma'$ can be obtained from $\Gamma$ by either $e_i/\overline{e}$ or $\overline{e}_i/\overline{e}$, then $\Gamma'$ contains a strict monotone cycle.
\end{lemma}
\begin{proof} We will use the same idea as in the proof of the previous lemma. First, if $s=0$ or if $n=1$, then there is no edge $e_i$ to slide over $e$. Therefore $\Gamma=\Gamma'$ and $e$ is a strict monotone cycle of one single edge in $\Gamma'$. 

Next we consider the case where $n=2$, then $\Gamma$ is a wedge sum of two oriented loops. We denote $e_i$ by $f$. By Remark \ref{remark: redundant cycles}, we have $\gamma$ is of the form $(f,...,f,e)$.
\begin{itemize}[topsep=0pt, itemsep=1ex]
\item[-] For $e_i/\overline{e}$, we have $\lambda(\overline{e})|\lambda(f)$. Since $\gamma$ is a strict monotone cycle, we have $\lambda(f)|\lambda(\overline{e})$. This implies that $\lambda(\overline{e})=\lambda(f)$. Moreover, we have $\lambda'(f)=\lambda(e)$. If $s=1$, then we have $\lambda(e)$ properly divides $\lambda(\overline{e})q(f)$. This implies that $\lambda(e)|\lambda(\overline{f})$ and $\lambda(e)\neq \lambda(\overline{f})$. It follows that $f$ is a strict virtually ascending loop in $\Gamma'$, which is a strict monotone cycle; if $s>1$, then we have $\lambda(f)|\lambda(\overline{e})q(f)$, i.e. $\lambda(f)|\lambda(\overline{f})$. It follows that $\lambda(\overline{f})=\lambda(f)h$ for some $h\in \mathbb{Z}_{>0}$. By Remark \ref{remark: redundant cycles}, we have $h\neq 1$. This implies that $\lambda(\overline{f})=\lambda(\overline{e})h$. Thus one can check the cycle $(\overline{e},f)$ is a strict monotone cycle in $\Gamma'$.
\item[-] For $\overline{e}_i/\overline{e}$, we have $\lambda(\overline{e})|\lambda(\overline{f})$. Since $\gamma$ is a strict monotone cycle, we have $\lambda(f)|\lambda(\overline{e})$. It follows that $\lambda(f)|\lambda(\overline{f})$. By Remark \ref{remark: redundant cycles}, we have $\lambda(f)\neq \lambda(\overline{f})$. Thus one can check the edge cycle $(e,f)$ is a strict monotone cycle in $\Gamma'$.

\end{itemize}

Now we consider the case where $n=3$, then $\Gamma$ is a wedge sum of three oriented loops. We denote $e_i$ by $f$ and the third loop by $g$. If $e_t\in \{f,\overline{f}\}$ for all $t$, this is covered by the 2-rose graph. Otherwise we can rewrite $\gamma$ as
$$\gamma=(e_1,...,e_{i_1-1},e_{i_1},e_{i_1+1},...,e_{i_k-1},e_{i_k},e_{i_k+1},...,e_s,e)$$
where each $e_{i_j}\in \{f,\overline{f}\}$ and $e_r\in \{g,\overline{g}\}$ for all $j=1,...,k$ and $r\neq i_j$.

For $e_i/\overline{e}$, we have $\lambda(\overline{e})|\lambda(f)$. Suppose that $e_{i_1}=f$, since $\gamma$ is a strict monotone cycle, we have $\lambda(f)|\lambda(\overline{e})q(e_1,...,e_{i_1-1})$. This implies that $\lambda(\overline{e})|\lambda(\overline{e})q(e_1,...,e_{i_1-1})$. Then we have the following cases:
\begin{itemize}[topsep=0pt, itemsep=1ex]
\item[-] If $i_1>1$, by Remark \ref{remark: redundant cycles}, we have $\lambda(\overline{e})\neq \lambda(\overline{e})q(e_1,...,e_{i_1-1})$, otherwise  the edge cycle $(e_1,...,e_{i_1-1})$ is a redundant subcycle of $\gamma$. This implies that either $q(g)$ or $q(\overline{g})$ is a non-trivial integer. Thus we have either $g$ or $\overline{g}$ is a strict monotone cycle in $\Gamma'$.
\item[-] If $i_1=1$, then we have $\lambda(\overline{e})=\lambda(f)$ and $\lambda'(f)=\lambda(e)$. Then we have the following two subcases:
\begin{itemize}[topsep=0pt, itemsep=1ex]
\item[-] If $k=1$, we claim that $(e_2,...,e_s,f)$ is a strict monotone cycle in $\Gamma'$. We notice that $(e_2,...,e_s)$ is a $\overline{f}$-edge path and the resulting label of $\overline{f}$ after sliding over it is $\lambda(\overline{f})q(e_2,...,e_s)=\lambda(\overline{e})q(e_1,...,e_s)$. Since $\gamma$ is a strict monotone cycle, we have $\lambda(e)$ properly divides $\lambda(\overline{e})q(e_1,...,e_s)$. It follows that $\lambda'(f)$ properly divides $\lambda(\overline{f})q(e_2,...,e_s)$. This shows that $(e_2,...,e_s,f)$ is a strict monotone cycle in $\Gamma'$.
\item[-] If $k>1$, then there exists $e_{i_2}\in \{f,\overline{f}\}$ in the edge path $(e_2,...,e_s)$. In addition, $(e_2,...,e_{i_2-1})$ is a $\overline{f}$-edge path in $\Gamma'$ and the resulting label of $\overline{f}$ after sliding over it is $\lambda(\overline{f})q(e_2,...,e_{i_2-1})$. 

If $e_{i_2}=f$, one can check the edge cycle $(e_2,...,e_{i_2-1},\overline{e},f)$ is a strict monotone cycle in $\Gamma'$.

On the other hand, if $e_{i_2}=\overline{f}$, then we have $\lambda(\overline{f})|\lambda(\overline{f})q(e_2,...,e_{i_2-1})$. Thus we have either $q(g)$ or $q(\overline{g})$ is a non-trivial integer. This implies that either $g$ or $\overline{g}$ is a strict monotone cycle in $\Gamma'$.
\end{itemize}
\end{itemize}
On the other hand, suppose that  $e_{i_1}=\overline{f}$. Since $\gamma$ is a strict monotone cycle, we have $\lambda(\overline{f})|\lambda(\overline{e})q(e_1,...,e_{i_1-1})$. Recall that $\lambda(\overline{e})|\lambda(f)$, then $\lambda(\overline{f})|\lambda(f)q(e_1,...,e_{i_1-1})$. Since there exists a $j$ such that $e_{i_j}=f$, otherwise there is no $f$ in $\gamma$, then we have $k>1$. Thus there exists $e_{i_2}\in \{f,\overline{f}\}$ in the edge subpath $(e_{i_1+1},...,e_s)$.
\begin{itemize}[topsep=0pt, itemsep=1ex]
\item[-] If $\lambda(\overline{f})$ properly divides $\lambda(f)q(e_1,...,e_{i_1-1})$, one can check $(e,e_1,...,e_{i_1-1},\overline{f})$ is a strict monotone cycle in $\Gamma'$.
\item[-] Otherwise we have $\lambda(\overline{f})=\lambda(f)q(e_1,...,e_{i_1-1})$, then $\lambda(\overline{f})=\lambda(\overline{e})q(e_1,...,e_{i_1-1})$. \\
If $e_{i_2}=f$, then $\lambda(f)|\lambda(\overline{e})q(e_1,...,e_{i_2-1})$. It follows that $\lambda(f)|\lambda(f)q(e_{i_1+1},...,e_{i_2-1})$. By Remark \ref{remark: redundant cycles}, we have $q(e_{i_1+1},...,e_{i_2-1}) \neq 1$. This implies that either $q(g)$ or $q(\overline{g})$ is a non-trivial integer. Thus we have either $g$ or $\overline{g}$ is a strict monotone cycle in $\Gamma'$.\\
On the other hand, if $e_{i_2}=\overline{f}$, then we have $\lambda(\overline{f})|\lambda(f)q(e_{i_1+1},...,e_{i_2-1})$, i.e. $\lambda(\overline{f})h=\lambda(f)q(e_{i_1+1},...,e_{i_2-1})$. By Remark \ref{remark: redundant cycles}, we have $h\neq 1$. This implies that $(e,e_{i_1+1},...,e_{i_2-1},\overline{f})$ is a strict monotone cycle in $\Gamma'$.
\end{itemize}

For $\overline{e}_i/\overline{e}$, we have $\lambda(\overline{e})|\lambda(\overline{f})$. We first notice that if there exists $e_{i_j}=\overline{f}$, then $\Gamma'$ is obtained from $\Gamma$ by $e_{i_j}/\overline{e}$. By the previous argument, we know that $\Gamma'$ contains a strict monotone cycle. So it suffices to consider $e_{i_j}=f$ for all $j=1,...,k$. Since $\gamma$ is a strict monotone cycle, we have $\lambda(f)|\lambda(\overline{e})q(e_1,...,e_{i_1-1})$. This implies that $\lambda(f)|\lambda(\overline{f})q(e_1,...,e_{i_1-1})$. Then we have the following cases:
\begin{itemize}[topsep=0pt, itemsep=1ex]
\item[-] If $\lambda(f)$ properly divides $\lambda(\overline{f})q(e_1,...,e_{i_1-1})$, one can check $(e, e_1,...,e_{i_1-1},f)$ is a strict monotone cycle in $\Gamma'$.
\item[-] Otherwise we have $\lambda(f)=\lambda(\overline{f})q(e_1,...,e_{i_1-1})$, then $\lambda(\overline{e})=\lambda(\overline{f})$ and $\lambda(f)=\lambda(\overline{e})q(e_1,...,e_{i_1-1})$. This implies that $\lambda(\overline{e})q(e_1,...,e_{i_1})=\lambda(\overline{e})q(e_1,...,e_{i_1-1})\frac{\lambda(\overline{f})}{\lambda(f)}=\lambda(\overline{f})=\lambda(\overline{e})$. It follows that $q(e_1,...,e_{i_1-1})=1$. By Remark \ref{remark: redundant cycles}, this case can not happen. 
\end{itemize}
\end{proof}

\begin{example} We provide an example to show that Lemma \ref{lemma: 3 rose slide over mobile edge} fails for $n$-rose graph with $n>3$. The labeled graphs in the figure below represent the same GBS group. We denote the labeled graph on the left (resp. right) by $\Gamma$ (resp. $\Gamma'$). $\Gamma$ has a strict monotone cycle $(f_3,f_2,f_4,f_1)$ with modulus 3 and $\Gamma'$ is obtained from $\Gamma$ by sliding $f_4$ over $\overline{f}_1$. 

We claim that $\Gamma'$ does not have any strict monotone cycles. To prove our claim, it suffices to show that there does not exist any $\overline{e}$-edge path $\alpha$ in $\Gamma'$ such that $\lambda'(e)|\lambda'(\overline{e})q'(\alpha)$ for any edge $e$ in $\Gamma'$. First, note that we can only slide $f_1,\overline{f}_1$ and $f_4$ in $\Gamma'$. 
\begin{itemize}[topsep=0pt, itemsep=1ex]
\item[-]For sliding $f_1$, the only possible $f_1$-edge path is $f_4$. However we have $30\nmid 14 q'(f_4)$. 
\item[-]For sliding $\overline{f}_1$, every $\overline{f}_1$-edge path consists of $\{f_2,\overline{f}_2,f_3,\overline{f}_3\}$. Since the prime factor 7 only appears in $\lambda'(f_1),\lambda'(f_4)$ and $\lambda'(\overline{f}_4)$, then $30q'(\alpha)$ can not be divisible by 14 for any $\overline{f}_1$-edge path $\alpha$.
\item[-] For sliding $f_4$, since $\lambda'(f_4)$ is only divisible by $\lambda'(f_1)$, then the first slide move of $f_4$ we can perform is $f_4/f_1$. Thus it suffices to show that there does not exist any $f_4$-edge path $\alpha$ in $\Gamma$ such that $21|30q(\alpha)$. Suppose there exists a such $\alpha$, then we have $q(\alpha)=q(f_1)^{\alpha_1}q(f_2)^{\alpha_2}q(f_3)^{\alpha_3}=(\frac{15}{7})^{\alpha_1}(\frac{5}{2})^{\alpha_2}(\frac{4}{5})^{\alpha_3}$ for some $\alpha_1,\alpha_2,\alpha_3\in \mathbb{Z}$. Since the prime factor 7 only appears on $\lambda(f_1)$ and $\lambda(\overline{f}_4)$, then we must have $\alpha_1<0$, otherwise $30q(\alpha)$ is not divisible by 21. In addition, since 30 only contains one prime factor 3 and only 15 in the above decomposition of $q(\alpha)$ contains prime factor 3. Thus, to ensure $30q(\alpha)$ is an integer, we must have $\alpha_1=-1$. It follows that $30q(\alpha)$ doesn't contain any prime factor 3. Therefore $21\nmid 30q(\alpha)$, which is a contradiction.
\end{itemize}

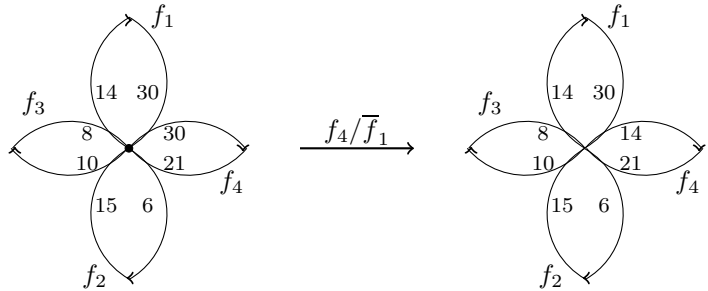
\begin{center}
\begin{tikzpicture}
\begin{scope}[decoration={
    markings,
    mark=at position 0.5 with {\arrow[scale=1.5]{>}}}
    ]
      \node[label={below, yshift=-0.3cm:}] at (0.7,3.05) {\scriptsize$10$};
      \node[label={below, yshift=-0.3cm:}] at (0.7,3.45) {\scriptsize$8$};
      \node[label={below, yshift=-0.3cm:}] at (1.85,3.05) {\scriptsize$21$};
      \node[label={below, yshift=-0.3cm:}] at (1.85,3.45) {\scriptsize$30$};
      \node[label={below, yshift=-0.3cm:}] at (0.95,4) {\scriptsize$14$};
      \node[label={below, yshift=-0.3cm:}] at (1.5,4) {\scriptsize$30$};
      \node[label={below, yshift=-0.3cm:}] at (1.5,2.5) {\scriptsize$6$};
      \node[label={below, yshift=-0.3cm:}] at (0.95,2.5) {\scriptsize$15$};
      \node[label={below, yshift=0.3cm:}] at (1.7,5) {\small$f_1$};
      \node[label={below, yshift=0.3cm:}] at (0.8,1.55) {\small$f_2$};
      \node[label={below, yshift=0.3cm:}] at (0,3.85) {\small$f_3$};
      \node[label={below, yshift=0.3cm:}] at (2.6,2.8) {\small$f_4$};

      \node[label={below, yshift=-0.3cm:}] at (6.7,3.05) {\scriptsize$10$};
      \node[label={below, yshift=-0.3cm:}] at (6.7,3.45) {\scriptsize$8$};
      \node[label={below, yshift=-0.3cm:}] at (7.85,3.05) {\scriptsize$21$};
      \node[label={below, yshift=-0.3cm:}] at (7.85,3.45) {\scriptsize$14$};
      \node[label={below, yshift=-0.3cm:}] at (6.95,4) {\scriptsize$14$};
      \node[label={below, yshift=-0.3cm:}] at (7.5,4) {\scriptsize$30$};
      \node[label={below, yshift=-0.3cm:}] at (7.5,2.5) {\scriptsize$6$};
      \node[label={below, yshift=-0.3cm:}] at (6.95,2.5) {\scriptsize$15$};
      \node[label={below, yshift=0.3cm:}] at (7.7,5) {\small$f_1$};
      \node[label={below, yshift=0.3cm:}] at (6.8,1.55) {\small$f_2$};
      \node[label={below, yshift=0.3cm:}] at (6,3.85) {\small$f_3$};
      \node[label={below, yshift=0.3cm:}] at (8.6,2.8) {\small$f_4$};

      \node[label={below, yshift=0.3cm:}] at (4.25,3.5) {\small$f_4/\overline{f}_1$};

      \tikzset{enclosed/.style={draw, circle, inner sep=0pt, minimum size=.1cm, fill=black}}
      
      \node[enclosed, label={right, yshift=.2cm:}] at (1.25,3.25) {};
      \node[label={right, yshift=-.2cm:}] at (1.25,4.25) {};

      \draw(1.25, 3.25) arc(240:120:1);
      \draw(1.25, 4.99) arc(60:-60:1);
      \draw(1.25, 3.25) arc(120:240:1);
      \draw(1.25, 1.51) arc(-60:60:1);
      \draw(1.25, 3.25) arc(140:40:1);
      \draw(1.25, 3.25) arc(40:140:1);
      \draw(2.8, 3.25) arc(-40:-140:1);
      \draw(-0.3, 3.25) arc(-140:-40:1);

      \draw(7.25, 3.25) arc(240:120:1);
      \draw(7.25, 4.99) arc(60:-60:1);
      \draw(7.25, 3.25) arc(120:240:1);
      \draw(7.25, 1.51) arc(-60:60:1);
      \draw(7.25, 3.25) arc(140:40:1);
      \draw(7.25, 3.25) arc(40:140:1);
      \draw(8.8, 3.25) arc(-40:-140:1);
      \draw(5.7, 3.25) arc(-140:-40:1);

      \draw[->, thick] (1.28,4.95) -- (1.29,4.95);
      \draw[->, thick] (2.75,3.22) -- (2.75,3.21);
      \draw[->, thick] (1.23,1.53) -- (1.22,1.53);
      \draw[->, thick] (-0.25,3.26) -- (-0.25,3.27);
      \draw[->, thick] (7.28,4.95) -- (7.29,4.95);
      \draw[->, thick] (8.75,3.22) -- (8.75,3.21);
      \draw[->, thick] (7.23,1.53) -- (7.22,1.53);
      \draw[->, thick] (5.75,3.26) -- (5.75,3.27);
      \draw[->, thick] (3.5, 3.25) -- (5, 3.25);

      \end{scope}
\end{tikzpicture}
\end{center}
\end{example}

\subsection{Small-rose GBS groups}
\begin{prop}\label{prop: smc}
Let $G$ be a $n$-rose GBS group such that the underlying $n$-rose graph $\Gamma$ satisfies either one of the following conditions:
\begin{itemize}[topsep=0pt, itemsep=1ex]
\item[(a)] $\Gamma$ has only one mobile edge; or
\item[(b)] $n\leq 3$.
\end{itemize}
If $\Gamma$ has a strict monotone cycle and $\Gamma'\in$ RLG($G$), then $\Gamma'$ also has a strict monotone cycle.
\end{prop}
\begin{proof}
By Theorem \ref{theorem: 3 moves}, we only need to consider when $\Gamma$ and $\Gamma'$ are related by a slide, induction or $\mathscr{A}^{\pm1}$-move. Recall that an induction move preserves the involved strict ascending loop. An $\mathscr{A}^{+ 1}$-move (resp. $\mathscr{A}^{-1}$-move) takes a strict virtually ascending loop (resp. strict ascending loop) to a strict ascending loop (resp. strict virtually ascending loop). Therefore, if $\Gamma$ and $\Gamma'$ are related by an induction or $\mathscr{A}^{\pm 1}$-move, then they both contain strict monotone cycles.

Now assume that $\Gamma$ contains a strict monotone cycle $\gamma=(e_1,...,e_s,e)$ and that $\Gamma'$ is obtained by sliding an edge $f$ in $\Gamma$. Clearly, if $f\notin E(\gamma)$, then $\gamma$ is also a strict monotone cycle in $\Gamma'$. If $f\in E(\gamma)$, we first consider part (a). By Lemma 3.17 in \cite{MM08}, we know that $\Gamma'$ can not be obtained from $\Gamma$ by $f/e$ or $f/\overline{e}$. Thus, by Lemma \ref{lemma: slide over non-mobile edge}, we have $\Gamma'$ also has a strict monotone cycle. Part (b) follows directly by Lemma \ref{lemma: slide over non-mobile edge}, Lemma \ref{lemma: n rose slide over mobile edge} and Lemma \ref{lemma: 3 rose slide over mobile edge}.
\end{proof}

\begin{cor}\label{coro: ascending=smc}
Let $G$ be a $n$-rose GBS group such that the underlying $n$-rose graph $\Gamma$ satisfies either one of the following conditions:
\begin{itemize}[topsep=0pt, itemsep=1ex]
\item[(a)] $\Gamma$ has only one mobile edge; or
\item[(b)] $n\leq 3$.
\end{itemize}
Then $G$ is ascending if and only if $\Gamma$ has a strict monotone cycle.
\end{cor}
\begin{proof}
If $G$ is ascending, then there exists $\Gamma'\in$ RLG($G$) that contains a strict ascending loop. By Theorem \ref{theorem: 3 moves}, there exists a finite sequence
$$\Gamma=\Gamma_0\rightarrow \Gamma_1\rightarrow...\rightarrow \Gamma_k=\Gamma'$$
such that each $\Gamma_i\rightarrow \Gamma_{i+1}$ is either a slide, an induction or an $\mathscr{A}^{\pm 1}$-move. Let $\Gamma_l$ be the first labeled graph in the sequence that contains a strict monotone cycle. If $l=0$, then we have $\Gamma=\Gamma_l$ and it has a strict monotone cycle; if $l>0$, we consider the preceding move $\Gamma_{l-1}\rightarrow \Gamma_l$. Since induction moves and $\mathscr{A}^{\pm 1}$-moves require both labeled graphs contain strict monotone cycles, then $\Gamma_{l-1}\rightarrow \Gamma_l$ must be a slide move. Moreover, the subsequence $\Gamma_0\rightarrow...\rightarrow \Gamma_l$ must be a sequence of slide moves too. We notice that there is only one vertex in $\Gamma$, so any labeled graph that is obtained from $\Gamma$ by a sequence of slide moves is also a $n$-rose graph. In particular, $\Gamma_l$ is a $n$-rose graph. Under condition (a), by Corollary 3.15 in \cite{MM08}, we have $\Gamma_l$ also has only one mobile edge. Thus, by part (a) of Proposition \ref{prop: smc}, we have $\Gamma$ contains a strict monotone cycle. Under condition (b) we apply part (b) of Proposition \ref{prop: smc}. Thus under either condition (a) or (b), if $G$ is ascending then $\Gamma$ has a strict monotone cycle. The converse is Lemma 3.6 in \cite{MM08}.
\end{proof}

Now we will give a proof of our first main result:

\begin{theorem}\label{thm: determine ascending}
Let $G$ be a $n$-rose GBS group that satisfies either one of the following conditions:
\begin{itemize}[topsep=0pt, itemsep=1ex]
\item[(a)] the underlying $n$-rose graph $\Gamma$ has only one mobile edge; or
\item[(b)] $n\leq 3$.
\end{itemize} 
Then there is an algorithm to determine whether $G$ is ascending.
\end{theorem}
\begin{proof}
This follows directly by Proposition \ref{prop: smc algorithm } and Corollary \ref{coro: ascending=smc}.
\end{proof}

\begin{example}\label{example: 3-rose non-ascending}
Let $p$ be a positive integer corpime to 7, and let $G$ be a GBS group represented by the following labeled graph. We apply the algorithm in Theorem \ref{thm: determine ascending} (i.e. the algorithm in Proposition \ref{prop: smc algorithm }) to show that $G$ is non-ascending.

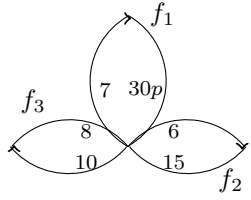
\begin{center}
\begin{tikzpicture}
\begin{scope}[decoration={
    markings,
    mark=at position 0.5 with {\arrow[scale=1.5]{>}}}
    ]
      \node[label={below, yshift=-0.3cm:}] at (0.7,3.05) {\scriptsize$10$};
      \node[label={below, yshift=-0.3cm:}] at (0.7,3.45) {\scriptsize$8$};
      \node[label={below, yshift=-0.3cm:}] at (1.85,3.05) {\scriptsize$15$};
      \node[label={below, yshift=-0.3cm:}] at (1.85,3.45) {\scriptsize$6$};
      \node[label={below, yshift=-0.3cm:}] at (0.95,4) {\scriptsize$7$};
      \node[label={below, yshift=-0.3cm:}] at (1.48,4) {\scriptsize$30p$};
      \node[label={below, yshift=0.3cm:}] at (1.7,5) {\small$f_1$};
      \node[label={below, yshift=0.3cm:}] at (0,3.85) {\small$f_3$};
      \node[label={below, yshift=0.3cm:}] at (2.6,2.8) {\small$f_2$};

      \draw(1.25, 3.25) arc(240:120:1);
      \draw(1.25, 4.99) arc(60:-60:1);
      \draw(1.25, 3.25) arc(140:40:1);
      \draw(1.25, 3.25) arc(40:140:1);
      \draw(2.8, 3.25) arc(-40:-140:1);
      \draw(-0.3, 3.25) arc(-140:-40:1);

      \draw[->, thick] (1.28,4.95) -- (1.29,4.95);
      \draw[->, thick] (2.75,3.22) -- (2.75,3.21);
      \draw[->, thick] (-0.25,3.26) -- (-0.25,3.27);
      \end{scope}
\end{tikzpicture}
\end{center}

First, there is no strict ascending loop or strict virtually ascending loop in the above labeled graph. Next, we note that we can only slide $\overline{f}_1$. In addition, $(f_3,f_2)$ is a strict $\overline{f}_1$-integer cycle with modulus 2. Thus $\overline{f}_1$ is the only mobile edge in the above labeled graph. Now it suffices to show that there is no strict monotone cycle with the last edge $f_1$. This is true because there is no 2-tuple of integers $(\alpha_2,\alpha_3)$ that satisfies $7|30p\cdot(\frac{4}{5})^{\alpha_2}(\frac{5}{2})^{\alpha_3}$.

\end{example}

\section{Non-ascending GBS groups}
In this section, we will work out some methods to rearrange slide moves in a non-ascending deformation space. In particular, we will show that any two slide moves of mobile edges of a $n$-rose graph with $n\leq 3$ in a non-ascending deformation space \textit{commute}, in the sense that the slides can be performed in the opposite order after possibly modifying the slide paths. In order to prove this, our idea is to study several slide equivalences and then use the proof of Proposition 3.19 in \cite{MM08}, which shows that slides of mobile edges commute with slides of non-mobile edges. We observe that, if we restrict to non-ascending $n$-rose GBS groups with $n\leq 3$, our slide equivalences either can not happen or they happen the same as slide equivalences in \cite{MM08}. Thus the proof of Proposition 3.19 in \cite{MM08} can be applied in our case. For simplicity, we only consider slide moves in $\mbox{RLG}$($G^{+}$).

\begin{notation}
If $\Gamma\in$ RLG($G$), $e,f\in E(\Gamma)$ and $A$ (resp. $B$) is an $e$ (resp. $f$)-edge path, we will use the notation $e/A\cdot f/B$ to denote the slide moves of $e$ over $A$ followed by sliding $f$ over $B$. In addition, we let $\emptyset$ indicate that the slide move cannot happen.
\end{notation}

\begin{definition}\cite[Definition 3.20]{MM08}
After applying slide moves, we sometimes implement a \textit{renaming} step. This means that if we have a sequence of slide moves, the subsequence after the renaming step should be renamed. For example, suppose we have a sequence of slide moves of the form $e/A\cdot \overline{f}/B\cdot \overline{e}/f$. Further, suppose we have a renaming step $e\mapsto f, f\mapsto \overline{e}$ after the first slide move. Then that sequence will turn into $e/A\cdot e/B\cdot \overline{f}/\overline{e}$.
\end{definition}

\begin{lemma}\label{lemma: slide relation}
Let $\Gamma\in$ RLG($G^+$) representing a non-ascending GBS group $G$. Suppose that $e$ and $f$ are two distinct mobile edges in $\Gamma$. Then we have the following slide equivalences:
\begin{itemize}[topsep=0pt, itemsep=1ex]
     \item[(a)] $e/A\cdot f/B=f/B\cdot e/A$
     \item[(b)] $e/f\cdot f/B=f/B\cdot e/Bf$
     \item[(c)] $e/f\cdot \overline{f}/B=\overline{f}/B\cdot e/f\overline{B}$
     \item[(d)] $e/A\cdot f/e=f/\overline{A}e\cdot e/A$
     \item[(e)] $e/A\cdot f/\overline{e}=f/\overline{e}A\cdot e/A$
     \item[(f)] $e/f\cdot f/e=\emptyset$
     \item[(g)] $e/f\cdot f/\overline{e}=\emptyset$
     \item[(h)] $e/f\cdot \overline{f}/\overline{e}=\emptyset$
     \item[(i)] $e/f\cdot \overline{f}/e=f/e$, then rename $e\mapsto \overline{f}, f\mapsto e$
     \item[(j)] $e/C\cdot \overline{e}/D=\overline{e}/D\cdot e/C$
\end{itemize}
where $f,\overline{f}, e,\overline{e}\not\in A, B$.
\end{lemma}
\begin{proof}
Note that both $A$ and $B$ don't contain $e,\overline{e},f$ or $\overline{f}$, so case (a) is valid. Similarly, since both $C$ and $D$ don't contain $e$ or $\overline{e}$, then case (j) is also valid. For case (b), (c), (d) and (e), it is straightforward to check that each pair of slide moves results in the same labeled graph.

To see that case (f) is not valid, we have $f$ is a loop at $o(e)$. For simplicity, we let $k=\lambda(e)$, $k'=\lambda(\overline{e})$, $s=\lambda(f)$ and $s'=\lambda(\overline{f})$. Thus we have the following diagram:
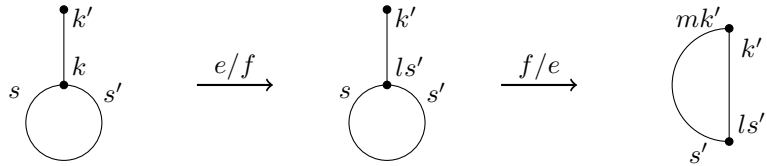
\begin{center}
\begin{tikzpicture}
\begin{scope}[decoration={
    markings,
    mark=at position 0.5 with {\arrow[scale=1.5]{>}}}
    ]
      \node[label={below, yshift=-0.3cm:}] at (0.6,3.15) {\small$s$};
      \node[label={below, yshift=-0.3cm:}] at (1.9,3.15) {\small$s'$};
      \node[label={left, yshift=0.3cm:}] at (1.45,3.5) {\small$k$};
      \node[label={left, yshift=0.3cm:}] at (1.5,4.15) {\small$k'$};
      \node[label={below, yshift=0.3cm:}] at (3.5,3.5) {\small$e/f$};
      \node[label={below, yshift=-0.3cm:}] at (4.95,3.15) {\small$s$};
      \node[label={below, yshift=-0.3cm:}] at (6.15,3.15) {\small$s'$};
      \node[label={left, yshift=0.3cm:}] at (5.8,3.5) {\small$ls'$};
      \node[label={left, yshift=0.3cm:}] at (5.75,4.15) {\small$k'$};
      \node[label={below, yshift=0.3cm:}] at (7.5,3.5) {\small$f/e$};
      \node[label={left, yshift=0.3cm:}] at (9.6,4.15) {\small$mk'$};
      \node[label={left, yshift=0.3cm:}] at (9.6,2.35) {\small$s'$};
      \node[label={left, yshift=0.3cm:}] at (10.3,3.75) {\small$k'$};
      \node[label={left, yshift=0.3cm:}] at (10.3,2.75) {\small$ls'$};
      
      \tikzset{enclosed/.style={draw, circle, inner sep=0pt, minimum size=.1cm, fill=black}}
      
      \node[enclosed, label={right, yshift=.2cm:}] at (1.25,3.25) {};
      \node[enclosed, label={right, yshift=-.2cm:}] at (1.25,4.25) {};
      \node[enclosed, label={right, yshift=.2cm:}] at (5.5,3.25) {};
      \node[enclosed, label={right, yshift=-.2cm:}] at (5.5,4.25) {};
      \node[enclosed, label={right, yshift=.2cm:}] at (10,4) {};
      \node[enclosed, label={right, yshift=.2cm:}] at (10,2.5) {};

      \draw(1.25,3.25) -- (1.25, 4.25) node[midway, right] {};
      \draw(1.75, 2.75) arc(0:360:0.5);
      \draw(5.5,3.25) -- (5.5, 4.25) node[midway, right] {};
      \draw(6, 2.75) arc(0:360:0.5);
      \draw(10,2.5) -- (10, 4) node[midway, right] {};
      \draw(10, 4) arc(90:270:0.75);

      \draw[->, thick] (3, 3.25) -- (4, 3.25);
      \draw[->, thick] (7, 3.25) -- (8, 3.25);

      \end{scope}
\end{tikzpicture}
\end{center}
Let $\Gamma'$ denote the middle labeled graph in the above diagram. If $e$ can slide over $f$ in $\Gamma$, then $k=ls$ for some $l\in \mathbb{Z}_{>0}$. If $f$ can then slide over $e$ in $\Gamma'$, then we have $mls'=s$ for some $m\in \mathbb{Z}_{>0}$. If $ml\neq 1$, then $f$ is a strict virtually ascending loop in $\Gamma'$. After an $\mathscr{A}$-move, $f$ will become a strict ascending loop. This contradicts with the fact that $G$ is non-ascending. On the other hand, if $ml=1$, then $s'=s$. Since $f$ is a mobile edge, by Remark \ref{remark: non-sacending} there exists either a strict $f$-integer cycle or a strict $\overline{f}$-integer cycle in $\Gamma'$. We consider the case that there exists a strict $f$-integer cycle $\gamma$, the latter case can be addressed by similar reasoning. Now we can turn $f$ into a strict virtually ascending loop by sliding $f$ over $\gamma$. Therefore, after an $\mathscr{A}$-move, $f$ will become a strict ascending loop, which is a contradiction.

For case (g), we have $e$ is a loop at $o(f)$. Thus we have the following diagram:
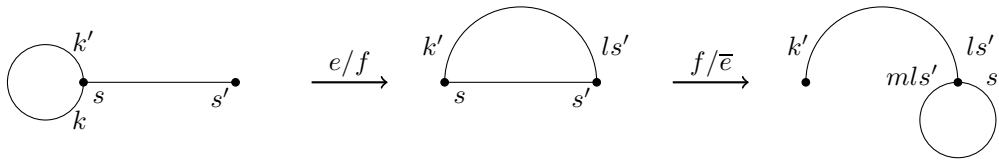
\begin{center}
\begin{tikzpicture}
\begin{scope}[decoration={
    markings,
    mark=at position 0.5 with {\arrow[scale=1.5]{>}}}
    ]
      \node[label={below, yshift=-0.3cm:}] at (0.2,3.05) {\small$s$};
      \node[label={below, yshift=-0.3cm:}] at (1.8,3.05) {\small$s'$};
      \node[label={left, yshift=0.3cm:}] at (-0.05,2.75) {\small$k$};
      \node[label={left, yshift=0.3cm:}] at (0,3.8) {\small$k'$};
      \node[label={below, yshift=0.3cm:}] at (3.5,3.5) {\small$e/f$};
      \node[label={below, yshift=-0.3cm:}] at (4.95,3.05) {\small$s$};
      \node[label={below, yshift=-0.3cm:}] at (6.55,3.05) {\small$s'$};
      \node[label={left, yshift=0.3cm:}] at (7,3.75) {\small$ls'$};
      \node[label={left, yshift=0.3cm:}] at (4.6,3.75) {\small$k'$};
      \node[label={below, yshift=0.3cm:}] at (8.25,3.5) {\small$f/\overline{e}$};
      \node[label={left, yshift=0.3cm:}] at (10.9,3.35) {\small$mls'$};
      \node[label={left, yshift=0.3cm:}] at (12,3.35) {\small$s'$};
      \node[label={left, yshift=0.3cm:}] at (9.4,3.75) {\small$k'$};
      \node[label={left, yshift=0.3cm:}] at (11.8,3.75) {\small$ls'$};
      
      \tikzset{enclosed/.style={draw, circle, inner sep=0pt, minimum size=.1cm, fill=black}}
      
      \node[enclosed, label={right, yshift=.2cm:}] at (0,3.25) {};
      \node[enclosed, label={right, yshift=-.2cm:}] at (2,3.25) {};
      \node[enclosed, label={right, yshift=.2cm:}] at (4.75,3.25) {};
      \node[enclosed, label={right, yshift=-.2cm:}] at (6.75,3.25) {};

      \node[enclosed, label={right, yshift=.2cm:}] at (9.5,3.25) {};
      \node[enclosed, label={right, yshift=-.2cm:}] at (11.5,3.25) {};

      \draw(0,3.25) -- (2, 3.25) node[midway, right] {};
      \draw(0, 3.25) arc(0:360:0.5);
      \draw(4.75, 3.25) arc(180:0:1);
      \draw(4.75,3.25) -- (6.75, 3.25);
      \draw(9.5, 3.25) arc(180:0:1);
      \draw(12, 2.75) arc(360:0:0.5);

      \draw[->, thick] (3, 3.25) -- (4, 3.25);
      \draw[->, thick] (7.75, 3.25) -- (8.75, 3.25);

      \end{scope}
\end{tikzpicture}
\end{center}
The argument is similar to case (f). In other words, we can always turn $f$ in the last labeled graph above into a strict ascending loop, which is a contradiction.

For case (h), we have the following diagram:
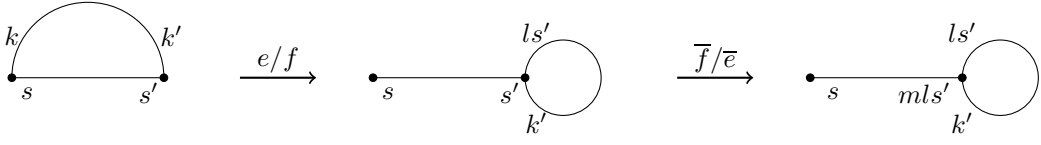
\begin{center}
\begin{tikzpicture}
\begin{scope}[decoration={
    markings,
    mark=at position 0.5 with {\arrow[scale=1.5]{>}}}
    ]
      \node[label={below, yshift=-0.3cm:}] at (0.2,3.05) {\small$s$};
      \node[label={below, yshift=-0.3cm:}] at (1.8,3.05) {\small$s'$};
      \node[label={left, yshift=0.3cm:}] at (2.1,3.8) {\small$k'$};
      \node[label={left, yshift=0.3cm:}] at (0,3.8) {\small$k$};
      \node[label={below, yshift=0.3cm:}] at (3.5,3.5) {\small$e/f$};
      \node[label={below, yshift=-0.3cm:}] at (4.95,3.05) {\small$s$};
      \node[label={below, yshift=-0.3cm:}] at (6.55,3.05) {\small$s'$};
      \node[label={left, yshift=0.3cm:}] at (6.9,3.85) {\small$ls'$};
      \node[label={left, yshift=0.3cm:}] at (6.9,2.65) {\small$k'$};
      \node[label={below, yshift=0.3cm:}] at (9.25,3.5) {\small$\overline{f}/\overline{e}$};
      \node[label={left, yshift=0.3cm:}] at (10.8,3.05) {\small$s$};
      \node[label={left, yshift=0.3cm:}] at (12,3.05) {\small$mls'$};
      \node[label={left, yshift=0.3cm:}] at (12.5,2.65) {\small$k'$};
      \node[label={left, yshift=0.3cm:}] at (12.5,3.85) {\small$ls'$};
      
      \tikzset{enclosed/.style={draw, circle, inner sep=0pt, minimum size=.1cm, fill=black}}
      
      \node[enclosed, label={right, yshift=.2cm:}] at (0,3.25) {};
      \node[enclosed, label={right, yshift=-.2cm:}] at (2,3.25) {};
      \node[enclosed, label={right, yshift=.2cm:}] at (4.75,3.25) {};
      \node[enclosed, label={right, yshift=-.2cm:}] at (6.75,3.25) {};

      \node[enclosed, label={right, yshift=.2cm:}] at (10.5,3.25) {};
      \node[enclosed, label={right, yshift=-.2cm:}] at (12.5,3.25) {};

      \draw(0,3.25) -- (2, 3.25) node[midway, right] {};
      \draw(0, 3.25) arc(180:0:1);
      \draw (7.75,3.25) arc (0:360:0.5); 
      \draw(4.75,3.25) -- (6.75, 3.25);
      \draw(10.5,3.25) -- (12.5, 3.25) node[midway, right] {};
      \draw(13.5, 3.25) arc(360:0:0.5);

      \draw[->, thick] (3, 3.25) -- (4, 3.25);
      \draw[->, thick] (8.75, 3.25) -- (9.75, 3.25);

      \end{scope}
\end{tikzpicture}
\end{center}
Let $\Gamma'$ (resp. $\Gamma''$) denote the middle (resp. last) labeled graph in the above diagram. If $e$ can slide over $f$ in $\Gamma$, then we have $k=ls$ for some $l\in \mathbb{Z}_{>0}$. If $\overline{f}$ can then slide over $\overline{e}$ in $\Gamma'$, then we have $s'=mk'$ for some $m\in \mathbb{Z}_{>0}$. This implies that $\lambda''(e)=mlk'$. Similar to case (f), we can always turn the edge $e$ in $\Gamma''$ into a strict ascending loop, which is a contradiction.

For case (i), we have the following diagram:
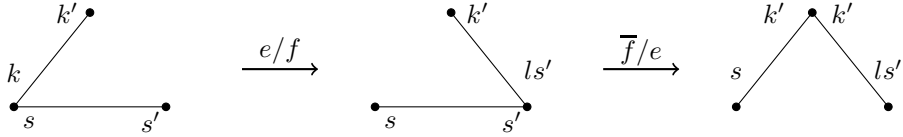
\begin{center}
\begin{tikzpicture}
\begin{scope}[decoration={
    markings,
    mark=at position 0.5 with {\arrow[scale=1.5]{>}}}
    ]
      \node[label={below, yshift=-0.3cm:}] at (0.2,3.05) {\small$s$};
      \node[label={below, yshift=-0.3cm:}] at (1.8,3.05) {\small$s'$};
      \node[label={left, yshift=0.3cm:}] at (0.7,4.5) {\small$k'$};
      \node[label={left, yshift=0.3cm:}] at (0,3.7) {\small$k$};
      \node[label={below, yshift=0.3cm:}] at (3.5,4) {\small$e/f$};
      \node[label={below, yshift=-0.3cm:}] at (4.95,3.05) {\small$s$};
      \node[label={below, yshift=-0.3cm:}] at (6.55,3.05) {\small$s'$};
      \node[label={left, yshift=0.3cm:}] at (6.9,3.75) {\small$ls'$};
      \node[label={left, yshift=0.3cm:}] at (6.1,4.5) {\small$k'$};
      \node[label={below, yshift=0.3cm:}] at (8.25,4) {\small$\overline{f}/e$};
      \node[label={left, yshift=0.3cm:}] at (9.5,3.7) {\small$s$};
      \node[label={left, yshift=0.3cm:}] at (10,4.5) {\small$k'$};
      \node[label={left, yshift=0.3cm:}] at (10.9,4.5) {\small$k'$};
      \node[label={left, yshift=0.3cm:}] at (11.5,3.75) {\small$ls'$};
      
      \tikzset{enclosed/.style={draw, circle, inner sep=0pt, minimum size=.1cm, fill=black}}
      
      \node[enclosed, label={right, yshift=.2cm:}] at (0,3.25) {};
      \node[enclosed, label={right, yshift=-.2cm:}] at (2,3.25) {};
      \node[enclosed, label={right, yshift=.2cm:}] at (1,4.5) {};
      \node[enclosed, label={right, yshift=.2cm:}] at (4.75,3.25) {};
      \node[enclosed, label={right, yshift=.2cm:}] at (5.75,4.5) {};
      \node[enclosed, label={right, yshift=-.2cm:}] at (6.75,3.25) {};

      \node[enclosed, label={right, yshift=.2cm:}] at (9.5,3.25) {};
      \node[enclosed, label={right, yshift=.2cm:}] at (10.5,4.5) {};
      \node[enclosed, label={right, yshift=-.2cm:}] at (11.5,3.25) {};

      \draw(0,3.25) -- (2, 3.25) node[midway, right] {};
      \draw(0,3.25) -- (1, 4.5) node[midway, right] {};
      \draw(5.75,4.5) -- (6.75, 3.25) node[midway, right] {};
      \draw(4.75,3.25) -- (6.75, 3.25);
      \draw(9.5,3.25) -- (10.5, 4.5) node[midway, right] {};
      \draw(11.5,3.25) -- (10.5, 4.5) node[midway, right] {};

      \draw[->, thick] (3, 3.75) -- (4, 3.75);
      \draw[->, thick] (7.75, 3.75) -- (8.75, 3.75);

      \end{scope}
\end{tikzpicture}
\end{center}
Let $\Gamma'$ (resp. $\Gamma''$) denote the middle (resp. last) labeled graph in the above diagram. If $e$ can slide over $f$ in $\Gamma$, we have $k=ls$ and $\lambda'(e)=ls'$. If $\overline{f}$ can then slide over $e$ in $\Gamma'$, we have $ls'|s'$, which implies that $l=1$. Thus we have $k=s$ and $\lambda''(\overline{f})=k'$. The result of these two moves can also be achieved by sliding $f$ over $e$ in $\Gamma$, with the only difference being that, after $f/e$ in $\Gamma$, $f$ is in the position of $\overline{e}$ in $\Gamma''$, so later references to $\overline{e}$ should be renamed as $f$. Similarly, references to $f$ should be renamed as $e$.
\end{proof}

\begin{lemma}\label{lemma: sequence commute}
Let $\Gamma\in$ RLG($G^+$) representing a non-ascending GBS group $G$. Suppose that $e$ and $f$ are two distinct mobile edges in $\Gamma$. Then we have the following slide equivalences:
\begin{itemize}[topsep=0pt, itemsep=1ex]
     \item[(k)] $e/Af A'\cdot f/B=f/B\cdot e/ABfA'$
     \item[(l)] $e/A\overline{f}A'\cdot f/B=f/B\cdot e/A\overline{f}\ \overline{B}A'$
     \item[(m)] $e/A\cdot f/BeB'=f/B\overline{A}eB'\cdot e/A$
     \item[(n)] $e/A\cdot f/B\overline{e}B'=f/B\overline{e}AB'\cdot e/A$
     \item[(o)] $e/AfA'\cdot f/BeB'=\emptyset$
     \item[(p)] $e/AfA'\cdot f/B\overline{e}B'=\emptyset$
     \item[(q)] $e/A\overline{f}A'\cdot f/B\overline{e}B'=\emptyset$
\end{itemize}
where $f,\overline{f}, e,\overline{e}\not\in A, B$.
\end{lemma}
\begin{proof}
We will use Lemma \ref{lemma: slide relation} to prove each case individually.

For case (k), we have
\begin{align*}
e/AfA' \cdot f/B= &\ e/A\cdot e/f \cdot f/B \cdot e/A'\ & \mbox{by\ (a)}\\
= &\ e/A\cdot f/B \cdot e/Bf \cdot e/A'\ & \mbox{by\ (b)}\\
 =&\ f/B\cdot e/ABfA'\ & \mbox{by\ (a)}
\end{align*}

We will skip illustrations of case (l), case (m) and case (n) because these cases are very similar to case (k).

For case (o), we have
\begin{align*}
e/AfA' \cdot f/BeB'= &\ e/A\cdot e/f \cdot f/B \cdot e/A' \cdot f/e\cdot f/B'\ & \mbox{by\ (a)}\\
 =&\ e/A\cdot f/B\cdot e/Bf \cdot f/\overline{A'}e \cdot e/A'\cdot f/B'\ & \mbox{by\ (b) and (d)}\\
 =&\ f/B\cdot e/A\cdot e/B\cdot e/f \cdot f/\overline{A'}\cdot f/e \cdot f/B'\cdot e/A'\ & \mbox{by\ (a)}\\
 =&\ f/B\cdot e/AB \cdot f/\overline{A'}\cdot e/\overline{A'}\cdot e/f\cdot f/e \cdot f/B'\cdot e/A'\ & \mbox{by\ (b)}\\
 =&\ \emptyset\ &\mbox{by\ (f)}
\end{align*}

For case (p), we have 
\begin{align*}
e/AfA' \cdot f/B\overline{e}B'= &\ e/A\cdot e/f\cdot f/B \cdot e/A' \cdot f/\overline{e} \cdot f/B'\ & \mbox{by\ (a)}\\
= &\ e/A\cdot f/B\cdot e/B\cdot e/f\cdot f/\overline{e} \cdot f/A' \cdot e/A' \cdot f/B'\ & \mbox{by\ (b) and (e)}\\
 =&\ \emptyset  & \mbox{by\ (g)}
\end{align*}

For case (q), we have
\begin{align*}
e/A\overline{f}A'\cdot f/B\overline{e}B'= &\ e/A\cdot e/\overline{f}\cdot f/B \cdot e/A'\cdot f/\overline{e} \cdot f/B' \ & \mbox{by\ (a)}\\
= &\ e/A\cdot f/B\cdot e/\overline{f}\ \overline{B}\cdot f/\overline{e}A' \cdot e/A' \cdot f/B' \ & \mbox{by\ (c) and (e)}\\
= &\ f/B\cdot e/A \cdot e/\overline{f}\cdot e/\overline{B}\cdot f/\overline{e}\cdot f/A'\cdot f/B' \cdot e/A' \ & \mbox{by\ (a)}\\
= &\ f/B\cdot e/A \cdot e/\overline{f}\cdot f/\overline{e}\cdot f/\overline{B}\cdot e/\overline{B}\cdot f/A'\cdot f/B' \cdot e/A' \ & \mbox{by\ (e)}\\
 =&\ \emptyset\  & \mbox{by\ (h)}
\end{align*}
\end{proof}

\begin{lemma}\label{lemma: slides in two graphs}
Let $\Gamma\in$ RLG($G^+$) be a $n$-rose graph representing a non-ascending GBS group $G$ with $n\leq 3$. Suppose $e,f$ are two distinct mobile edges in $\Gamma$. Then we have the following slide equivalences:
\begin{itemize}[topsep=0pt, itemsep=1ex]
\item[(r)] 
$e/A\overline{f}A'\cdot f/BeB'=
\begin{cases}
\emptyset & \mbox{if}\ n=1 \\
f/A'\cdot \overline{f}/\overline{A}e\cdot e/A\cdot \overline{e}/B', \ \mbox{rename}\ e\mapsto f, f\mapsto \overline{e} & \mbox{if}\ 1<n\leq 3 
\end{cases} $

\end{itemize}
where $f,\overline{f}, e,\overline{e}\not\in A, A', B, B'$.

\end{lemma}
\begin{proof}
If $n=1$, we only have one geometric edge in $\Gamma$. Thus the above slide relation is $\emptyset$. If $n=2$, since the edge paths $A,A',B,B'$ don't contain $e,\overline{e},f,\overline{f}$, then we have $A,A',B,B'$ are empty edge paths. Thus the above slide relation becomes
$$e/\overline{f}\cdot f/e=\overline{f}/e,\  \mbox{rename}\ e\mapsto f, f\mapsto \overline{e}$$
This follows directly from case (i) in Lemma \ref{lemma: slide relation}. If $n=3$, $\Gamma$ only contains three geometric edges. We then denote the third one by $g$. Since $e,\overline{e},f,\overline{f}\not\in A', B$, by Remark \ref{remark: redundant cycles}, we have $A', B$ are either $g^{s}=(g,g,...,g)$ or $\overline{g}^{t}=(\overline{g},\overline{g},...,\overline{g})$ for some $s,t\in\mathbb{Z}_{>1}$. It follows from the proof of case (d) of Lemma 3.23 in \cite{MM08} that $q(B\overline{A'})\in \mathbb{Z}$. To keep this proof self-contained, we repeat the reasoning here. If $e$ can slide over $A\overline{f}A'$ in $\Gamma$, then $\lambda(f)|\lambda(e)q(A\overline{f})$. If $f$ can then slide over $Be$, then $\lambda(e)q(A\overline{f}A')|\lambda(f)q(B)$. This implies that, after sliding $f$ over $B$, we can slide it over $\overline{A'}$ and the resulting label is $\lambda(f)q(B\overline{A'})$. In addition, we have $\lambda(e)q(A\overline{f})|\lambda(f)q(B\overline{A'})$. It follows that $\lambda(f)|\lambda(f)q(B\overline{A'})$. Thus we have $q(B\overline{A'})=q(g)^{k}\in \mathbb{Z}$ for some $k\in \mathbb{Z}$. If $k=0$, then we have $A'=B$. This implies that $f/B=f/A'$. Thus the above slide relation becomes 
\begin{align*}
e/A\overline{f}A' \cdot f/A'eB'= &\ e/A\cdot e/\overline{f} \cdot f/A' \cdot e/A' \cdot f/e\cdot f/B'\ & \mbox{by\ (a)}\\
=&\ e/A\cdot f/A'\cdot e/\overline{f}\  \overline{A'} \cdot e/A'\cdot f/e\cdot f/B'\ & \mbox{by\ (c)}\\
=&\ e/A\cdot f/A'\cdot e/\overline{f}\cdot f/e\cdot f/B'\ & \mbox{by\ cancellation}\\
=&\ e/A\cdot f/A'\cdot \overline{f}/e \cdot \overline{e}/B'\ \  \mbox{rename}\ e\mapsto f,f\mapsto\overline{e}\ & \mbox{by\ (i)}\\
=&\ f/A'\cdot e/A\cdot \overline{f}/e \cdot \overline{e}/B'\ \  \mbox{rename}\ e\mapsto f,f\mapsto\overline{e}\ & \mbox{by\ (a)}\\
=&\ f/A'\cdot \overline{f}/\overline{A}e\cdot e/A\cdot \overline{e}/B', \ \mbox{rename}\ e\mapsto f, f\mapsto \overline{e}\ & \mbox{by\ (d)}
\end{align*} If $k\neq 0$, since $G$ is non-ascending, we must have $q(g)=1$. By Remark \ref{remark: redundant cycles}, this can not happen.

\end{proof}

\begin{prop}
\label{prop: e commute}
Let $\Gamma\in$ RLG($G^+$) be a $n$-rose graph representing a non-ascending GBS group $G$ with $n\leq 3$ and $e_1,...,e_k\in E(\Gamma)$ be the mobile edges of $\Gamma$ for some $0\leq k\leq 3$. Suppose $\Gamma'$ is obtained from $\Gamma$ by a sequence of slides of mobile edges. Then there exists a rearrangement of the above sequence
$$\Gamma=\Gamma_0\rightarrow...\rightarrow\Gamma_1\rightarrow...\rightarrow\Gamma_{k-1}\rightarrow...\rightarrow \Gamma_k=\Gamma'$$
such that the subsequences $\Gamma_{j-1}\rightarrow...\rightarrow \Gamma_{j}$ are slides of the geometric edge $e_j, \overline{e}_j$ only for $j=1,...,k$.
\end{prop}

\begin{proof} We will apply the proof of Proposition 3.19 in \cite{MM08}. To keep this exposition self-contained, we start with a rough idea of this proof, and then explain the reason why it can be applied in our case. Let $e,f\in E(\Gamma)$ where $f$ is non-mobile. To simplify the discussion, we denote slides of the form $e/A$ or $\overline{e}/A$ by $E$, and those of the form $e/A fA',e/A\overline{f}A',\overline{e}/A fA'$ or $\overline{e}/A\overline{f}A'$ by $E_F$ where $e,\overline{e},f,\overline{f}\not\in A, A'$. Likewise define the symbols $F$ and $F_E$. Given a valid slide sequence $e/\alpha\cdot f/\beta$, the \textit{complexity} is the number of slides of the form $E_FF_E$ after omitting the symbols $E,F$. If the complexity is 0, then we can apply Lemma 3.21 and Lemma 3.22 in \cite{MM08} to put the sequence into the form $f/\beta'\cdot \overline{f}/\beta''\cdot e/\alpha'\cdot \overline{e}/\alpha''$. If the complexity is nonzero, then we can apply Lemma 3.23 in \cite{MM08} to reduce complexity.

We now explain why the above method can be applied in our case. Let $e_s,e_t\in E(\Gamma)$ be two mobile edges. If $e_s/\alpha\cdot e_t/\beta$ is a valid slide sequence in RLG($G^{+}$), we use Lemma \ref{lemma: slide relation}, Lemma \ref{lemma: sequence commute} and Lemma \ref{lemma: slides in two graphs} to rearrange this slide sequence. Note that our slide equivalences in these three lemmas are either $\emptyset$ or the same as Lemma 3.21, Lemma 3.22 and Lemma 3.23 in \cite{MM08}. If one of the null slide equivalences happens, it will contradict the validity of the original sequence. Therefore, by the same reason as the proof of Proposition 3.19 in \cite{MM08}, we can always put the sequence into $e_t/\beta'\cdot \overline{e_t}/\beta''\cdot e_s/\alpha'\cdot \overline{e_s}/\alpha''$ by applying valid slide equivalences in Lemma \ref{lemma: slide relation}, Lemma \ref{lemma: sequence commute} and Lemma \ref{lemma: slides in two graphs} finitely many times. Thus, by repeated application of this fact, we can obtain the desired sequence.

\end{proof}

\section{The isomorphism problem}
The goal of this section is to prove our second main result, which can be restated as the following:
\begin{theorem}\label{theorem: main result 3} Let $\Gamma$ and $\Gamma'$ be two labeled graphs defining GBS groups $G$ and $G'$. Suppose $\Gamma$ is a $n$-rose graph with $n\leq 3$ and $G$ is non-ascending, then there exists an algorithm that determines whether $G$ and $G'$ are isomorphic.
\end{theorem}

We remark that, by Theorem \ref{thm: determine ascending}, there is an algorithm to determine whether a $n$-rose GBS group with $n\leq 3$ is non-ascending or not. We also remark that Example \ref{example: 3-rose non-ascending} is an infinite family of GBS groups for which the isomorphism problem can be solved using the algorithm in the above theorem. 

To prove the above theorem, we will introduce the non-mobile subgraph of a given labeled graph that plays a key role in our proof.

\begin{definition}\cite{MM08}
We denote by $\Gamma_{non}\subset \Gamma$ the subgraph obtained from $\Gamma$ by discarding the mobile edges and any vertices incident to a strict ascending loop. $\Gamma_{non}$ is then called the \textit{non-mobile subgraph} of $\Gamma$.
\end{definition}

\begin{definition}
We denote by $\mbox{S}_{NM}(\Gamma)\subseteq$ RLG($G$) the set of reduced labeled graphs obtained from $\Gamma$ by a sequence of slides of non-mobile edges in $\Gamma$.
\end{definition}
By Proposition 3.10 in \cite{MM08}, we know that $|\mbox{S}_{NM}(\Gamma)|$ is finite.

\textbf{Proof of Theorem \ref{theorem: main result 3}} First we make both $\Gamma$ and $\Gamma'$ reduced by performing collapse moves. We may assume that $\Gamma,\Gamma'\in$ RLG($G^{+}$), by Lemma 2.7 in \cite{MM08}, since the orientation homomorphisms of $\Gamma$ and $\Gamma'$ agree. Suppose that $G\cong G'$. By Corollary \ref{coro: slide sequence}, there is a sequence of slide moves taking $\Gamma$ to $\Gamma'$. In addition, by Corollary 3.24 in \cite{MM08}, this slide sequence may be chosen so that there exists an intermediate graph $\Gamma''\in$ RLG($G^+$)
$$\Gamma\rightarrow...\rightarrow \Gamma''\rightarrow...\rightarrow\Gamma'$$ such that the subsequence $\Gamma\rightarrow...\rightarrow\Gamma''$ is a family of slides of non-mobile edges only and the subsequence $\Gamma''\rightarrow...\rightarrow\Gamma'$ is a family of slides of mobile edges only. Since all non-mobile edges stay stationary in the subsequence $\Gamma''\rightarrow...\rightarrow \Gamma'$, then the non-mobile subgraph of $\Gamma''$ and $\Gamma'$ agree. Let $e_1,...,e_k$ be mobile edges of $\Gamma$ for some $0\leq k\leq 3$. By Corollary 3.15 in \cite{MM08}, they are mobile edges of $\Gamma''$ too. By Proposition \ref{prop: e commute}, there is a sequence $$\Gamma''=\Gamma_{0}\rightarrow...\rightarrow\Gamma_{1}\rightarrow...\rightarrow\Gamma_{k-1}\rightarrow...\rightarrow \Gamma_{k}=\Gamma'$$
such that the subsequences $\Gamma_{j-1}\rightarrow \Gamma_{j}$ are slides of the geometric edge $e_j,\overline{e}_j$ only for $j=1,...,k$. Thus our original slide sequence becomes
$$\Gamma\rightarrow...\rightarrow\Gamma''=\Gamma_{0}\rightarrow...\rightarrow\Gamma_{1}\rightarrow...\rightarrow\Gamma_{k-1}\rightarrow...\rightarrow \Gamma_{k}=\Gamma'$$

Now we establish an algorithm to find such a slide sequence.

First, search all labeled graphs in $\mbox{S}_{NM}(\Gamma)$.
\begin{itemize}[topsep=0pt, itemsep=1ex]
\item[-] If there is no labeled graph in $\mbox{S}_{NM}(\Gamma)$ that has the non-mobile subgraph isomorphic to $\Gamma'_{non}$, then $G\not\cong G'$.
\item[-] Otherwise we denote by $\mbox{S}^{A}_{NM}(\Gamma)$ the set of all labeled graphs in $\mbox{S}_{NM}(\Gamma)$ such that the non-mobile subgraph agrees with $\Gamma_{non}'$ and we run the algorithm below.
\end{itemize}

Let $\Gamma''\in \mbox{S}^{A}_{NM}(\Gamma)$. We now provide an algorithm to check that whether there exists a slide subsequence $\Gamma''=\Gamma_0\rightarrow...\rightarrow \Gamma_1$.
We denote the set of mobile edges in $\Gamma'$ by $\mbox{ME}(\Gamma')$. We check that whether there exists a mobile edge $e$ in $\Gamma''$ such that we can slide it with the resulting labels $\lambda(e)=\lambda'(f)$ and $\lambda(\overline{e})=\lambda'(\overline{f})$ for some $f\in \mbox{ME}(\Gamma')$. This can be done in the following way:

Let $e$ be a mobile edge in $\Gamma''$.
We provide a method to check that whether we can obtain $\lambda(e)=\lambda'(f)$, the other equation above can be checked by the same way. We first denote the petals of $\Gamma''$ by $g_1,...,g_n$. Without loss of generality, we may assume that $e=g_1$. Recall that, similarly to Proposition \ref{prop: smc algorithm }, we can associate a $r$-tuple of integers for each label in $\Gamma''$, where $r$ is the number of different prime factors of modulus of all petals in $\Gamma''$. We denote by $(\sigma_1,...,\sigma_r)$ the $r$-tuple of the label $\lambda''(e)$ of $e$ in $\Gamma''$. Since the slide move doesn't change the prime factors, we can also associate a $r$-tuple of integers for each label in $\Gamma'$. Given $f\in \mbox{ME}(\Gamma')$, we then denote by $(\tau_1,...,\tau_r)$ the $r$-tuple of $\lambda'(f)$. Recall that each element in the set $$\Lambda(\overline{e})=\bigcup_{i=1}^L\Bigl\{(y_2,...,y_n)\ |\  w^i_l\leq \sigma_l+\sum_{h=2}^n y_h\alpha_l^h\ \mbox{for all}\ l=1,...,r\Bigl\}$$
corresponds to a possible label $\lambda''(e)\prod_{h=2}^nq(g_h)^{y_h}$ of $e$ as a consequence of slide moves of $e$, where $L$ is finite, $(w_1^i,...,w_r^i)$ is a computable vector in $\mathbb{Z}_{\geq 0}^r$ for each $i$ and $(\alpha_1^h,...,\alpha_r^h)$ is the $r$-tuple of $q(g_h)$. It now suffices to check that whether there exists an element $(y_2,...,y_n)\in \Lambda(\overline{e})$ such that $\lambda''(e)\prod_{h=2}^nq(g_h)^{y_h}=\lambda'(f)$. This can be done by checking that whether $w_l^i\leq \tau_l$ for each $l=1,...,r$ and $i=1,...,L$, and that whether there exists a solution $(y_2,...,y_n)$ of the linear system $A(y_2,...,y_n)^{\intercal}=(\tau_1-\sigma_1,...,\tau_r-\sigma_r)^{\intercal}$ where $A=(\alpha_l^h)$ is a $r\times (n-1)$ integer matrix. Finally, by repeating the above process for every mobile edge in $\mbox{ME}(\Gamma')$, we can determine whether we can slide the mobile edge $e$ in $\Gamma''$ to obtain a desired label $\lambda(e)=\lambda'(f)$.

Now we have the following two cases:
\begin{itemize}[topsep=0pt, itemsep=1ex]
\item[-] If there exists such a mobile edge $e$, we then denote by $\Gamma_1$ the labeled graph obtained from $\Gamma''$ by sliding $e$ with the resulting labels $\lambda(e)=\lambda'(f)$ and $\lambda(\overline{e})=\lambda'(\overline{f})$. Note that $\Gamma_1$ not only has $\Gamma_{non}'$ as the non-mobile subgraph but also has the position and label of one mobile edge agreeing with $\Gamma'$. In addition, there exists a slide subsequence $\Gamma''=\Gamma_0\rightarrow ...\rightarrow 
\Gamma_1$. Now we can check that whether there exists the rest slide sequence $\Gamma_1\rightarrow...\rightarrow \Gamma_{k-1}\rightarrow...\rightarrow \Gamma_k$ by keep repeating the above algorithm $k$ times.
\item[-] If there is no mobile edge $e$ that satisfies the above criterion, then there is no slide sequence of mobile edges between the given $\Gamma''$ and $\Gamma'$. 
\end{itemize} 
Finally, we can repeat the above algorithm to each labeled graph in $\mbox{S}_{NM}^A(\Gamma)$. By the above argument, $G$ and $G'$ are isomorphic if and only if we can find a slide sequence of mobile edges between $\Gamma''$ and $\Gamma'$ for some $\Gamma''\in \mbox{S}^{A}_{NM}(\Gamma)$ and $\Gamma'$.

\end{document}